\documentclass[11pt,a4paper,reqno]{amsart}
\usepackage[utf8]{inputenc}
\usepackage{amsmath,amssymb,amsthm, mathtools}
\usepackage{enumerate}
\usepackage{xfrac}
\usepackage[symbol,perpage]{footmisc}
\usepackage{hyperref}
\usepackage{enumitem}
\usepackage[british]{babel}
\usepackage{graphicx, subcaption, xcolor, bbm}
\usepackage[margin=1in]{geometry}

\usepackage{tikz-cd}
\theoremstyle{plain}
\newtheorem{thm}{Theorem}[section]
\newtheorem{prop}[thm]{Proposition}
\newtheorem{lemma}[thm]{Lemma}

\newtheorem*{claim*}{Claim}
\newtheorem*{thm*}{Theorem}

\theoremstyle{definition}
\newtheorem{defi}[thm]{Definition}
\newtheorem*{defi*}{Definition}

\newtheorem*{notation*}{Notation}

\usepackage{chngcntr}
\counterwithin{figure}{section}

\theoremstyle{remark}

\newtheorem{obs}[thm]{Remark}
\newtheorem*{obs*}{Remark}

\newtheorem{rema}[thm]{Remark}
\newtheorem*{rema*}{Remark}

\theoremstyle{plain}
\newcounter{main}

 \newtheorem{maintheorem}[main]{Theorem}

\def\DD{\mathbb{D}}

\renewcommand{\setminus}{\smallsetminus}
\newcommand{\anna}[1]{\textcolor{magenta}{#1}}
\newcommand{\luke}[1]{\textcolor{blue}{#1}}

\title{Boundaries of Baker domains of entire functions. A finer approach}

\author[Anna Jov\'e]{Anna Jov\'e}
\address{
	 Departament de Matemàtiques i Informàtica, Universitat de Barcelona, Barcelona, Spain}
\email{annajove@ub.edu}

\author{{\L}ukasz Pawelec}
\address{Institute of Mathematical Economics, SGH Warsaw School of Economics, 
al.~Niepodleg\l{}o\'{s}ci~162, 02-554 Warszawa, Poland}
\email{lpawel@sgh.waw.pl}

\thanks{The first author acknowledges financial support from the Spanish government grants PID2023-147252NB-I00 and  FPI PRE2021-097372. The research of the second author was funded in whole or in part by the National Science Centre, Poland, grant no. 2023/49/B/ST1/03015.}

\date{\today}

\begin{document}
\begin{abstract}
We consider transcendental entire functions having doubly parabolic Baker domains, such that the Denjoy-Wolff point of the associated inner function is not a singularity. We describe in a very precise way the dynamics on the boundary from a measure-theoretical point of view. 
Applications of such results lead to a better understanding of the topology and the dynamics on the boundaries. In particular, we improve some of the results in \cite{FagellaJove} for the Baker domain of $z+e^{-z}$.
In fact, our conclusions are obtained by applying new results established here on the dynamics of the radial extension of one component doubly parabolic inner functions, which strengthen those of \cite{ivrii2023innerfunctionscompositionoperators}.
\end{abstract}
\maketitle

\section{Introduction}
Let $ f\colon\mathbb{C}\to\mathbb{C} $ be a transcendental entire function   and denote by $ \left\lbrace f^n\right\rbrace _{n\in\mathbb{N}} $ its iterates, which generate a discrete dynamical system in $ \mathbb{C} $. Then, the complex plane is divided into two totally invariant sets: the {\em Fatou set} $ \mathcal{F}(f) $, defined to be the set of points $ z\in\mathbb{C} $ such that $ \left\lbrace f^n\right\rbrace _{n\in\mathbb{N}} $ forms a normal family in some neighbourhood of $ z$; and the {\em Julia set} $ \mathcal{J}(f) $, its complement. The connected components of the Fatou set are called {\em Fatou components}.

For entire functions, invariant Fatou  components are always simply connected \cite{baker1984}, so the Riemann map can be used as a uniformization. More precisely, let $ U $ be a invariant Fatou component of $ f $ and let $ \varphi $ be a Riemann map from the open unit disk $ \mathbb{D} $ onto $ U $. Then,  \[
g\colon\mathbb{D}\longrightarrow\mathbb{D},\hspace{1cm}g\coloneqq \varphi^{-1}\circ f\circ\varphi
\] is an analytic self-map of $ \mathbb{D} $. The function $ g $ is called the {\em associated inner function}. Since  $ f|_{U} $ and $ g|_{\mathbb{D}} $ are conformally conjugate by $ \varphi $, the study of holomorphic self-maps of $ \mathbb{D} $ is a good approach to analyse the dynamics of $ f|_{U} $. Moreover, we can extend the conjugacy by using radial limits for both $ g $ and $ \varphi $, i.e.
\[g^*\coloneqq\lim\limits_{t\to 1^-}g(t\xi), \hspace{1cm}\varphi^*\coloneqq\lim\limits_{t\to 1^-}\varphi(t\xi),\] for $ \xi\in\partial\mathbb{D} $. Then, $ \varphi^*\circ g^*(\xi)=f\circ \varphi^* (\xi)$ for $ \lambda $-almost every $ \xi\in \partial\mathbb{D} $ \cite[Lemma 3.11]{JF23}, where $ \lambda $ denotes the normalized Lebesgue measure on $ \partial\mathbb{D} $. 

Let us focus on {\em doubly parabolic Baker domains}, i.e. invariant Fatou components in which iterates converge locally uniformly to $ \infty $, and for which the hyperbolic distance between iterates converges to zero. Baker domains are exclusive to transcendental functions, and present richer dynamics (both in the interior and on the boundary), affected by the presence of the essential singularity. 

Moreover, if $U$ is a doubly parabolic Baker domain, there are infinitely many accesses from $U$
\cite{BakerDom} (see also \cite{BFJK-Accesses, JF23}), which form a natural partition of the boundary $ \partial U $, in the following way.
\begin{defi*}{\bf (Target)}
Let $ f\colon \mathbb{C}\to\mathbb{C} $ be a transcendental entire function,	let $U$ be a Baker domain, and let $\varphi\colon\mathbb{D}\to U$ be a Riemann map. Let $I\subset \partial \mathbb{D}$ be a closed arc, delimited by  $\xi_1\neq  \xi_2\in\partial \mathbb{D}$ such that $ \varphi^*(\xi_1)=\varphi^*(\xi_2)=\infty$. We say that $ \varphi^*(I) \subset\partial U$ is a {\em target} on $\partial U$.
\end{defi*}

See Figure \ref{fig-radial-limits} for a visual intuition.

	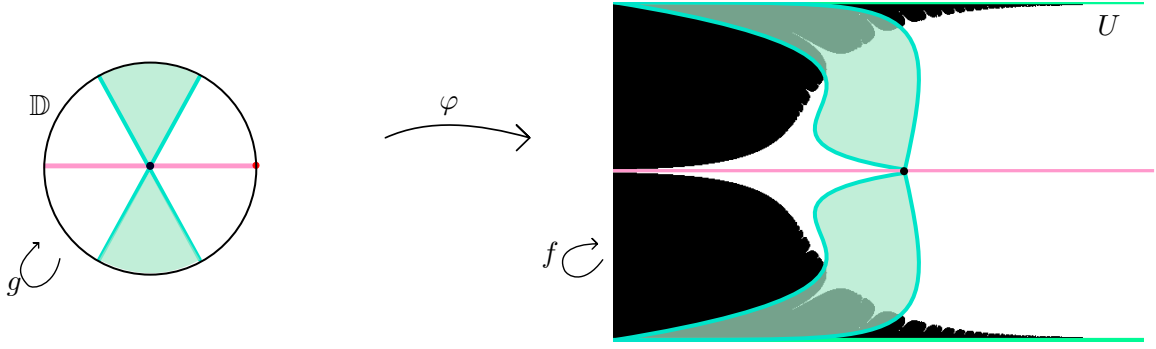
\begin{figure}[htb!]\centering
	\includegraphics[width=15cm]{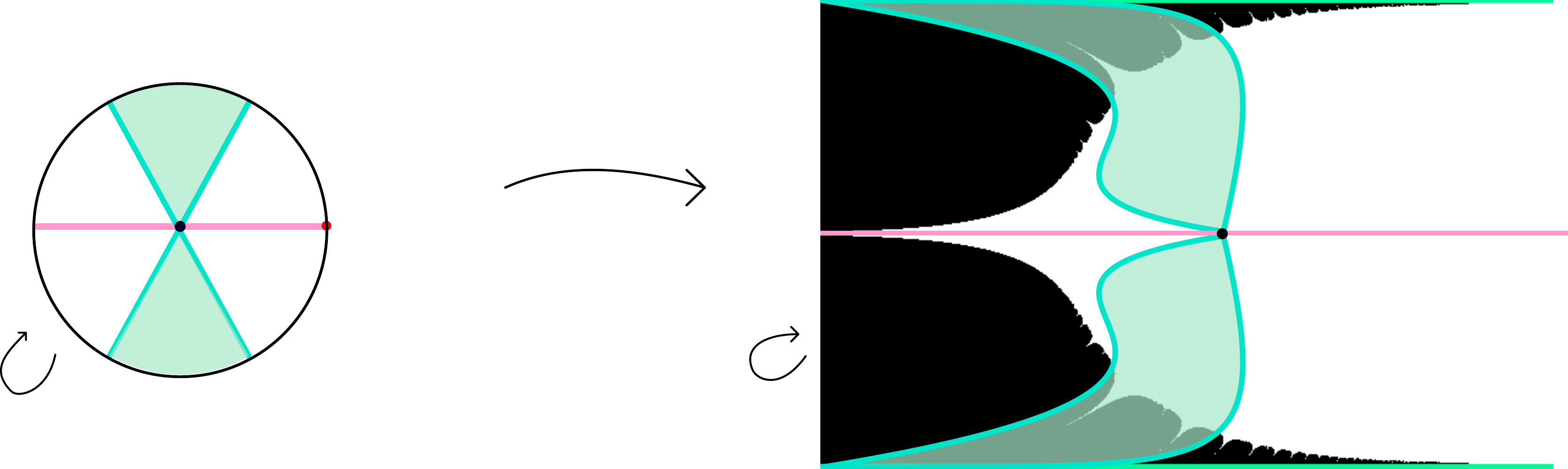}
	\setlength{\unitlength}{15cm}
	\put(-0.63, 0.205){$ \varphi $}
	\put(-1.02, 0.045){ $g $}
	\put(-0.54, 0.07){$f$}
		\put(-1, 0.2){ $\mathbb{D} $}
	\put(-0.05, 0.27){$U$}
	\caption{\footnotesize Schematic representation of targets. Indeed, for the Baker domain $U$ of $ f(z)=z+e^{-z} $, the figure shows the image of two arcs in the unit circle, each of them bounded by two rays landing at infinity in the dynamical plane.}
	\label{fig-radial-limits}
\end{figure}

Hence, given such an invariant Fatou component, it is a natural to examine the distribution of orbits on $\partial U$ with respect to a given target.

A first step on this direction are the results of Doering and Mañé \cite{DM91} (and also \cite[Thms. B and C]{BFJK-Escaping}), which give a description of $ f|_{\partial U} $ from a measure-theoretical point of view. More precisely, if $U$ is a doubly parabolic Baker domain, $ f|_{\partial U} $ is always ergodic with respect to any harmonic measure $\omega_U$. In the case when $ f|_{\partial U} $ has finite degree or the Denjoy-Wolff point of the associated inner function is not a singularity, then $ f|_{\partial U} $ is also conservative, and thus $ \omega_U $-almost every orbit is dense on $ \partial U $. 

In contrast with attracting basins, no harmonic measure supported on $\partial U$ is $ f $-invariant, so one cannot do a straightforward analysis applying classical results from ergodic theory, such as Birkhoff Ergodic Theorem or Kac's Lemma (compare with the results given in Appendix \ref{appendix}). 

Nevertheless, there exists an infinite $\sigma$-finite measure $ \mu $ which is $ f $-invariant and absolutely continous with respect to $\omega_U$ \cite{DM91}. This measure $ \mu $ allows us to prove more precise statements on the ergodic properties of $ f|_{\partial U} $ and the distribution of orbits with respect to targets, although the complexity of the arguments is much greater than in the case where an invariant probability measure exists.

\subsection*{Statement of results}
For a Baker domain, the convergence point of the interior orbits is placed on the boundary. In  terms of the associated inner function, this means that the {\em Denjoy-Wolff point} (the point of convergence of the interior orbits) is on the unit circle (and can be assumed to be 1, by precomposing the Riemann map $\varphi\colon\mathbb{D}\to U$ with an appropriate rotation).  Denote by $\omega_U$ the harmonic measure on $\partial U$ obtained as the push-forward of the Lebesgue measure $\lambda$ on $\partial\mathbb{D}$ under $\varphi^*$.

The $\sigma$-finite measure $ \mu $ on $\partial U$ which is invariant under $ f $ is defined as follows.
Denote by $\psi\colon\mathbb{H}\to U$ a conformal map from the upper-half plane to $U$, such that points in $\mathbb{H}$ converge to $\infty$ when iterating $h\coloneqq \psi^{-1}\circ f\circ\psi$. Radial limits are defined for $h$ and $\psi$ (denoted by $h^*$ and $\psi^*$, respectively).
Let $\mu$ the push-forward by $\psi^*$ of the Lebesgue measure  on $\mathbb{R}$, which is  $T$-invariant.  Thus, $\mu$ is $f$-invariant.

We make the following distinction between targets.
\begin{defi*}{\bf (Finite target and infinite target)}
		Let $U$ be a doubly parabolic Baker domain, and let $\psi\colon\mathbb{H}\to U$ be as above. Let $I\subset \mathbb{R}\cup \left\lbrace \infty\right\rbrace $ be a closed  interval, delimited by  $x_1\neq  x_2\in \mathbb{R}$ such that $ \psi^*(x_1)=\psi^*(x_2)=\infty$. 
		\begin{itemize}
			\item If $ I$ is bounded on $\mathbb{R}$, we say that $ \psi^*(I) \subset\partial U$ is a {\em finite target} on $\partial U$. 
			\item  If the complement of  $I$ is bounded on $\mathbb{R}$, we say that $ \psi^*(I) \subset\partial U$ is a {\em infinite target} on $\partial U$.
		\end{itemize}
\end{defi*}
See Figure \ref{fig-radial-limits2} for a visual intuition.
We prove the following.

\begin{maintheorem}\label{thm-a}
	Let $f$ be a transcendental entire function, and let $ U $ be a doubly parabolic Baker domain  such that the Denjoy-Wolff point of the associated inner function is not a singularity. 
\begin{enumerate}[label={\em (\alph*)}]
		\item  {\em (Occupation times)}\label{thmA.a}  Let $ E\subset\partial U $ be a finite target, and let\[ S_nE(x)= \# \left\lbrace 0\leq k\leq n-1: f^k(x)\in E \right\rbrace =\sum_{k=0}^{n} \mathbbm{ 1 }\circ f^k(x).\]   Then, for all $\varepsilon\in (0, \frac12)$,  \[\lim\limits_n \frac{S_nE(x)}{n^{\frac12-\varepsilon}}\to\infty, \] for $ \omega_U $-almost every $ x\in \partial U $. In particular, $ {S_n E}/{\sqrt[3]{n}} \to \infty$ $ \omega_U $-almost everywhere.
			\item {\em (First return times)}\label{thmA.b}  Let $ E\subset\partial U $ be a finite target, and let \[E_n\coloneqq \left\lbrace x\in E\colon \tau_E(x)=n\right\rbrace, \] where $ \tau_E(x) \in \mathbb{N}$ denotes the first return time of $ x\in E $ to $ E $. Then, $ \omega_U(E_n)\sim  \frac{1}{n\sqrt{n}} $.
	\item {\em (Waiting times)}\label{thmA.c} Let $ F\subset\partial U $ be an infinite target, $F=\psi^*(I)$, {and assume the complement of $I$ in $\mathbb{R}$ is large enough}. Let \[F_n\coloneqq \left\lbrace x\in F\colon f(x), \dots, f^{n-1}(x)\in F, f^{n}(x)\notin F\right\rbrace. \] Then, $ \omega_U(F_n)\sim \frac{1}{\sqrt{n}} $.
\end{enumerate}
\end{maintheorem}

\begin{figure}[h!]\centering
	\includegraphics[width=15cm]{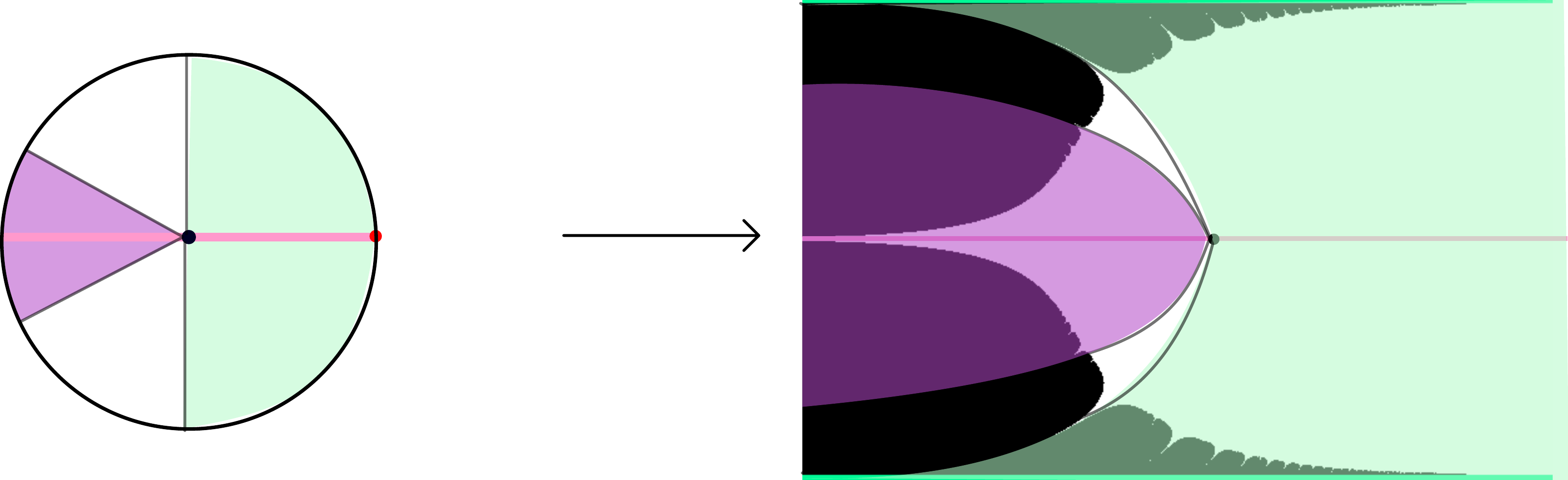}
	\setlength{\unitlength}{15cm}
	\put(-0.6, 0.17){$ \varphi $}
		\put(-0.75, 0.15){$ 1 $}
			\put(-1, 0.25){$ \mathbb{D} $}
		\put(-0.02, 0.27){$ U $}
	\caption{\footnotesize For the same Baker domain as above, a finite target (in purple) and an infinite target (in green). Note that 1 is the Denjoy-Wolff point of the associated inner function.}
	\label{fig-radial-limits2}
\end{figure}
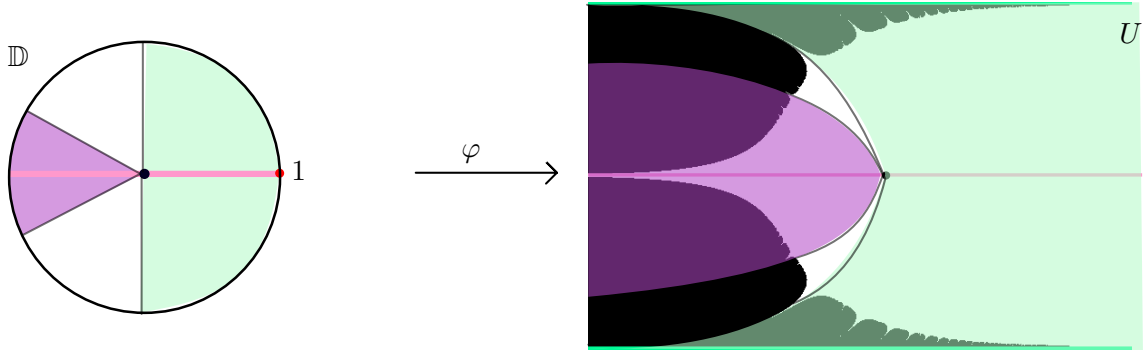

Inspired by results of Thaler and Zweimüller \cite{ThalerZweimuller} on distributional limit theorems in infinite ergodic theory, we prove the following finer results for Baker domains of finite degree (in fact, the assumption of finite degree can be weakened -- see Section \ref{subsect-proofAB}). 
\begin{maintheorem}\label{thm-b}
Let $f$ be a transcendental entire function, and let $ U $ be a doubly parabolic Baker domain  of finite degree. 
	\begin{enumerate}[label={\em (\alph*)}]
		\item\label{thmb-a} {\em (Darling-Kac theorem)} Let $E\subset \partial U$ be a finite target on $\partial U$, and let $$ S_nE(x)= \# \left\lbrace 0\leq k\leq n-1: f^k(x)\in E \right\rbrace .$$ Then, as $n\to\infty$,
		\[\omega_U\left(\left\lbrace \frac\pi{\sqrt {2n}} S_nE\leq \mu(E) t \right\rbrace  \right) \to \frac2\pi \int_0^t e^{-\frac{y^2}\pi }dy, \hspace{0.5cm} t\geq 0.\]
		\item\label{thmb-b} {\em (Arcsine law for occupation times)} 
		Let $ A\coloneqq \psi^*((0, +\infty)) $, and let $$ S_nA(x)= \# \left\lbrace 0\leq k\leq n-1: f^k(x)\in A \right\rbrace .$$	Then, as $n\to\infty$,
		\[\omega_U\left(\left\lbrace \frac1n S_nA\leq t \right\rbrace  \right) \to \frac2\pi \arcsin\sqrt t, \hspace{0.5cm} t\in [0,1].\]
			\item\label{thmb-c} {\em (Arcsine law for waiting times)}Let $E\subset \partial U$ be a finite target on $\partial U$, and let $Z_nE(x)$ be the time of the last visit of the orbit of $x$ to $E$ up to time $n$.
			Then, as $n\to\infty$,
			\[\omega_U\left(\left\lbrace \frac1n Z_nE\leq t \right\rbrace  \right) \to \frac2\pi \arcsin\sqrt t, \hspace{0.5cm} t\in [0,1].\]
	\end{enumerate}
\end{maintheorem}

\subsection*{Doubly parabolic inner functions}
	The  main theorems stated above are in fact obtained by first proving more general results for the radial extension of the associated inner functions, and then transfering them to the Baker domains, exploiting the compatibility between the targets and the Riemann map. 
	
Inner functions associated to doubly parabolic Baker domains are themselves doubly parabolic (Def. \ref{def:dblpar}). For these inner functions we develop a very precise measure-theoretic description of their boundary dynamics, including the following.
\begin{itemize}
	\item We give an asymptotic description of occupation times, first return times and waiting times (Thm. \ref{thm-a-in}), in the same spirit as in Theorem \ref{thm-a}.
	\item We consider a specific class of inner functions, namely {\em one component doubly parabolic inner functions} (Def. \ref{def-ivriiurbanski-onecomp}). We prove that the boundary of such inner functions fits into the framework of AFN-systems (Lemma \ref{lem:innerAFN}), which form a class of interval maps that preserve infinite measures, widely studied from the point of view of infinite ergodic theory. Thus, we will be able  to apply results on AFN-systems concerning the asymptotic distribution of orbits, thereby obtaining  finer recurrence properties of one-component inner functions (Thm. \ref{thm-b-in}) that translates into Theorem \ref{thm-b} above. 
	\item We prove the distribution of periodic points according to their period (Prop. \ref{prop-periodic-general}). More precisely, for $ \xi\in \partial \mathbb{D} $ and $r>0$, there exist a periodic point $ p_r\in D(\xi,r)$. The period of $p_r$ increases asymptotically at most as fast as 
	\[ \frac{-\ln r}{r^2}, \mbox{\quad as $ r\to 0$.}\]
 (Compare also with Theorem \ref{thm-dt} below.)
\end{itemize}

	The previous results for inner functions, presented in Sections \ref{sect-dpinnerfunctions} and \ref{section-ergodic-DP}, are of independent interest. 

\subsection*{Applications}
Even though the previous results are interesting on their own, we apply them to a  specific class of Baker domains to deduce some additional properties, which are interesting from a topological and dynamical point of view, and not necessarily related with ergodic theory.

Let us start by defining this specific class of Baker domains, which we call {\em expanding}.
\begin{defi*}{\bf (Expanding Baker domains)}
	Let $ f $ be a transcendental entire function, and $ U $ be a Baker domain. We say that $ U $ is an {\em expanding Baker domain}  if there exists $r>0$ such that the following properties are satisfied.
	\begin{enumerate}[label={(\alph*)}]
		\item $ U $ is doubly parabolic Baker domain, and the Denjoy-Wolff point of the associated inner function is not a singularity.
		\item  For every $x, y\in\partial U $ such that $ f(y)=x$ there exists a branch $ F_y$ of $ f^{-1} $ which is well-defined in $ D(x,r) $, and
	$F_y(D(x,r)\cap U)\subset U$.
	\item $ \left| f'(z)\right| >1 $, $ z\in D(x,r) $, for all $ x\in\partial U $; and for every finite target $ E $, $ \left| f'(z)\right| \geq K_E>1$, $ z\in D(x,r) $, for all $ x\in E$.
	\end{enumerate}
\end{defi*}

Although this seems a very restrictive condition, we provide a family of Baker domains of any finite degree satisfying the previous properties, inspired by  \cite{FH06}. In particular, the Baker domains of degree two of the function $ f(z) =z+e^{-z}$, previously studied in \cite{BakerDom, FH06, FagellaJove}, are expanding.

Given an expanding Baker domain, and a backward orbit $ \left\lbrace x_n\right\rbrace _n \subset\partial U$ (i.e. so that $ f(x_{n-1})=x_n $), the branch $F_n$ of $f^{-n}$ sending $x_0$ to $x_n$ is well-defined in $D(x_0,r)$, for all $n\geq 0$ (Lemma \ref{lemma-baker-inverse-branch}).  Applying Theorem \ref{thm-a}, we find precise bounds on the contraction for $\widetilde{\omega_U}$-almost every backward orbit (where $\widetilde{\omega_U}$ is the measure on the space of backwards orbits given by Rohklin's natural extension \cite{Rohlin}; see also \cite{jov2024pesintheorytranscendentalmaps}).

\begin{maintheorem}{\bf (Uniform subexponential contraction of inverse branches)}\label{thm-c}
	Let $ f $ be a transcendental entire function, and $ U $ be an expanding Baker domain. Then,  for $\widetilde{\omega_U}$-almost every backward orbit $ \left\lbrace x_n\right\rbrace _n $ and $r_0\in (0,r)$,  there exists $n_0\in\mathbb{N}$ and  $ K>1 $ such that, for all $n\geq n_0$,
	\[F_n(D(x_0, r_0))\subset D(x_n,  K^{-\sqrt[3]{n}} r_0),\] where $F_n$ is the branch of $f^{-n}$ sending $x_0$ to $x_n$.
\end{maintheorem}

It is well-known that periodic points are dense on the boundary of expanding Baker domains (compare with \cite{JF23, jove2024periodicboundarypointssimply}). However, standard proofs showing density of boundary periodic points \cite{Przytycki-Zdunik,FagellaJove, JF23, jove2024periodicboundarypointssimply, jov2024pesintheorytranscendentalmaps, jove2024boundarieshyperbolicsimplyparabolic} do not give control in the period of the periodic point. 

Combining the  techniques in  \cite{JF23, jove2024periodicboundarypointssimply} with our results on the inner functions (notably Proposition \ref{prop-periodic-general}), we prove the following result, which gives a finer description of the distribution of periodic points on the boundary of expanding Baker domains in terms of their period.

%\luke{I need to think how to phrase this better.}
%\begin{maintheorem}{\bf (Bound on the period of periodic boundary points)}\label{thm-d}
%Let $ f $ be an entire function and $ U $ be an expanding Baker domain of finite degree. Let $x\in\partial U$ and $ r>0 $. Then, for all $ r>0 $, there exist a periodic point $ p_r\in D(x,r) \cap \partial U$. The period of $ p_r $ increases asymptotically as 
%\[ \frac{-\ln \omega_U(D(x,r))}{\omega_U(D(x,r))^2},\] as $ r\to 0 $.
%\end{maintheorem}

%\anna{VERSION OF THE THEOREM WITH TARGETS}
%\begin{maintheorem}{\bf (Bound on the period of periodic boundary points)}\label{thm-dta}
%Let $ f $ be an entire function, $ U $ be an expanding Baker domain and let $ \varphi\colon\mathbb{H}\to U $ be a conformal map sending infinity to infinity. Let $z\in\partial U$, and let $ \left\lbrace E_n\right\rbrace _n $ be a collection of nested targets, $ \cap_n E_n=\emptyset$ and $z\in  \cap_n \overline{ E_n}$. Then, for all $ n\geq0 $, there exist a periodic point $ p_n\in E_n \cap \partial U$, where $ E_n=\varphi^*(I_n) $. Denote $ \left| I_n\right| =2r_n $ and $\varphi(x)=z$. The period of $ p_n $ increases asymptotically as 
%	\[ \frac{-\ln r_n}{r_n^2}+x^2,\] as $ n\to \infty $.
%\end{maintheorem}

\begin{maintheorem}{\bf (Bound on the period of periodic boundary points)}\label{thm-dt}
	Let $ f $ be an entire function, $ U $ be an expanding Baker domain and let $ \varphi\colon\mathbb{H}\to U $ be a conformal map sending infinity to infinity. Let $z\in\partial U$, and let $E= \varphi^*(I)$ be target, where $I=(x-r,x+r)$. Then, for all $r>0$ small enough, there exist a periodic point $ p_r\in E$. The period of $p_r$ increases asymptotically {at most as fast} as 
	\[ \frac{-\ln r}{r^2}+x^2, \mbox{\quad as $ r\to 0$.}\]
\end{maintheorem}

Finally, we apply the obtained results to describe the boundary of the Baker domain of the function $f(z)=z+e^{-z}$. The boundary of such Baker domain has been significantly studied in \cite{FagellaJove}, giving a detailed description of both the topology and the dynamics. More precisely, the boundary of such Baker domain is made out of hairs (curves of escaping points), some of them land at a point and others accumulate in  indecomposable continua, and that the dynamics of the hairs is encoded by some symbolic dynamics.  With the previously developed techniques we are able to prove that almost every of such hairs land, and we prove some estimates on the measure of codes satisfying some specific dynamical properties (see Theorem \ref{thm-e} for a precise statement).

\begin{obs*}
	We restricted ourselves to entire functions, since we need Fatou components to have infinitely many accesses to infinity (and preimages of infinity being dense in the unit circle), so that the concept of target makes sense and partititions the boundary in an interesting way. For meromorphic functions, this is no longer the case: there exist doubly parabolic Baker domains with a unique access to infinity, and their boundary may even be a Jordan curve \cite{BFJK-Accesses, BFJK-localcon}.
\end{obs*}

\subsection*{Structure of the paper}
We collect some needed background on Infinite Ergodic Theory  and AFN-systems in Section \ref{background-IET}, which will serve as a framework for the ergodic study of our maps. In Section \ref{sect-dpinnerfunctions}, we study doubly parabolic inner functions (the inner functions that can be associated to the Baker domains considered) and their mapping properties. In Section \ref{section-ergodic-DP} we prove a finer measure-theoretic description of the boundary dynamics of doubly parabolic inner functions (including estimates on the asymptotic distribution of orbits and periodic points).

Finally, in Section \ref{section-boundariesBD} we transfer all the results for doubly parabolic inner functions to the boundaries of Baker domains, using the targets and their compatibility with the Riemann map. The proof of the main theorems can be found in this section, even though they follow as direct  applications of the  work in the previous sections.

As mentioned earlier, the previous results have analogues in the case of basins of attraction (and, therefore, for centered inner functions), which can be obtained quite easily from classical results in ergodic theory (such as the Birkhoff Ergodic Theorem). We collect these results in Appendix \ref{appendix}, and invite the reader to consult them in order to appreciate the difficulties that arise intrinsically when an invariant probability measure is not available.

\begin{notation*}
	Throughout the paper, we write $ A\lesssim B $ to signify that $ A\leq CB $ for some constant $ C>0 $. If both $ A\lesssim B $ and $ B\lesssim A $, then we write $ A\sim B $. Given an interval $ I\subset \mathbb{R} $, we denote by $ \left| I \right|  $ its length.
\end{notation*}

\section{Background on Infinite Ergodic Theory}\label{background-IET} 
Let $ (X, \mathcal{A},\mu) $ be a measure space, and $T\colon X\to X$ be a measure-preserving transformation. We are interested in the case when $ \mu $ is an infinite $\sigma$-finite measure.

The main difficulty when dealing with measure-preserving transformations for which the preserved measure is not finite, but just $\sigma$-finite, is that the standard results in ergodic theory (such as the Birkhoff Ergodic Theorem or Kac's Lemma; see Sect. \ref{appendix-erg-th}) fail to be true. We refer e.g.  to \cite{Aaronson97}, \cite{notesInfErgodicTheory}, \cite[Chap. 10]{URM1} for basic background in this field, and we state here only the results that we need.

When dealing with infinite measure spaces, the use of sweep-out sets (which can be viewed as finite-measure building blocks of the system)  is particularly useful as it provides a finite-measure framework to work with.

\begin{defi}{\bf (Sweep-out set)}
	Let $T\colon X\to X$ be a measure-preserving transformation of a $\sigma$-finite measure space $ (X, \mathcal{A}, \mu) $. A set $A\in \mathcal{A}$, with $0<\mu(A)<\infty$ is a {\em sweep-out set} for the transformation $T$ if 
	\[\mu\Big(X\smallsetminus\bigcup_{n=0}^\infty T^{-n}(A)\Big)=0.\]
\end{defi}

Note that, if the transformation $T$ is ergodic and conservative, then every set of positive measure is a sweep-out set \cite[Corol. 10.1.14]{URM1}. The following lemma, which is well-known (compare e.g. \cite{notesInfErgodicTheory}), relates first return times and sweep-out sets.

\begin{lemma}{\bf (Return times to sweep-out sets)}\label{lemma-zweimuller}
	Let $T\colon X\to X$ be a measure-preserving transformation of a $\sigma$-finite measure space $ (X, \mathcal{A}, \mu) $. Let $A\in \mathcal{A}$ be a sweep-out set, $A^C=X\smallsetminus A$, then
	\[\mu (A\cap \left\lbrace \tau_A>n \right\rbrace)=\mu(A^C\cap\left\lbrace \tau_A=n \right\rbrace ),\] where $ \tau_A $ denotes the first return time to $ A $. 
\end{lemma}
{We now introduce a quantity, the wandering rate, that captures the cumulative time spent outside the sweep-out set before returning.
\begin{defi}{\bf (Wandering rate)}
		For a sweep-out set $A$,  the \emph{wandering rate} $w_n(A)$ is defined as 
	\[w_n(A)=\sum_{k=0}^{n-1}\mu(A\cap \{\tau_A >k\}).\]
\end{defi}}

The following pointwise version of an ergodic theorem applies to measure-preserving transformations of $ \sigma $-finite measure spaces.
\begin{thm}{\bf (Hopf Ergodic Theorem, {\normalfont \cite[Thm. 10.3.3]{URM1}})}\label{hopf-ergodic-thm}
	Let $T\colon X\to X$ be an ergodic and conservative measure-preserving transformation of a $\sigma$-finite measure space $ (X, \mathcal{A}, \mu) $. If $f,g\in L^1(\mu)$ and $\int_Xg d\mu\neq 0$, then 
	\[\lim\limits_n \frac{ \sum_{k=0}^{n-1} f(T^k(x))}{ \sum_{k=0}^{n-1} g(T^k(x))}=\frac{\int_Xf d\mu}{\int_Xg d\mu}, \hspace{0.5cm}\textrm{for }\mu\textrm{-a.e. }x\in X.\]
\end{thm}

In the particular case when $f$ and $g$ are the indicator functions of two measurable sets $E,F\subset X$ of positive finite measure, we get
\[\lim\limits_n\frac{ \# \left\lbrace 0\leq k\leq n-1: T^k(x)\in E \right\rbrace}{ \# \left\lbrace 0\leq k\leq n-1: T^k(x)\in F \right\rbrace}=\frac{\mu(E)}{\mu(F)},\]i.e. the occupation times are asympotically comparable.

Finally, let $L^1_+(\mu)$ denote the class of functions  $f\in L^1(\mu) $ such that $f\geq0$ and $ \int_Xfd\mu<\infty $. We need the following result by Aaronson.

\begin{thm}{\bf (Rescaling sequences, {\normalfont{\cite[Thm. 2.4.1]{Aaronson97}, \cite[Thm. 10.3.6]{URM1}}})}\label{thm-aaronson}
	Let $T\colon X\to X$ be an ergodic and conservative measure-preserving transformation of a $\sigma$-finite measure space $ (X, \mathcal{A}, \mu) $. Let $\left\lbrace a(n)\right\rbrace _n$ be a sequence such that 
	\[a(n)\nearrow\infty, \hspace{0.4cm}\frac{a(n)}n\searrow 0, \hspace{0.4cm}\textrm{as } n\to\infty.\]
	If there exists $A\in \mathcal{A}$ such that $0<\mu(A)<\infty$ and $ \int_A a(\tau_A(x))d\mu(x)<\infty $, where $\tau_A$ denotes the first return time to $A$, then
	\[\lim\limits_n\frac1{a(n)}\sum_{k=0}^{n-1} f(T^k(x)) =\infty \hspace{0.4cm} \textrm{ for }\mu\textrm{-a.e. }x\in X, \textrm{ for all }f\in L^1_+(\mu).\]
\end{thm}

In our applications (Section \ref{section-ergodic-DP}), $ a(n) $ is of the form $n^\alpha$, for some $\alpha\in (0,1)$ and $f$ is the indicator function of some measurable set $E\subset X$ of positive measure. This way we will determine that the occupation times of the set $E$ grow slower that the sequence $n^\alpha$, almost everywhere.

While there is no hope of pointwise ergodic theorems for infinite ergodic systems, it is often possible to find a limit in a weaker sense. We  use the following definition.

\begin{defi}{\bf (Strong distributional convergence)}
	Take a sequence of measurable random functions $R_n$ (taking values in $\mathbb{R}\cup{\pm\infty}$) defined on a $\sigma$-finite measure space $ (X, \mathcal{A}, \mu) $. 
We denote the distributional convergence of $(R_n)$ w.r.t. $\nu$ as $R_n \stackrel{\nu}{\Longrightarrow} R$.
	
\noindent	We say that $R$ is a \emph{strong distributional limit} of the sequence $R_n$ if $R_n \stackrel{\nu}{\Longrightarrow} R$ for any probability measure $\nu \ll \mu$. We denote this by $R_n \stackrel{\mathcal{L}(\mu)}{\Longrightarrow} R$.
\end{defi}

\subsection{Basic AFN-systems}\label{subsec:AFN}

An important family of infinite measure-preserving dynamical systems is given by piecewise-$ \mathcal{C}^2 $ interval maps with indifferent (neutral) fixed points, introduced by Thaler and Zweimuller \cite{ThalerZweimuller}. As explained in the introduction, we will prove that the maps that we consider are AFN-systems, and thus we can apply Thaler and Zweimuller's results, which we reproduce here. 

The definition of a basic AFN-system is as follows (by $ \lambda $ we denote here the Lebesgue measure on $ \mathbb{R} $).

\begin{defi}{\bf (Basic AFN-system)}\label{def:AFN}
Let $X\subset \mathbb{R}$ be a bounded interval, and let $\mathcal{I}=(I_k)_k$ be a countable collection of nonempty disjoint subintervals of $ X $, such that $\lambda(X \backslash\bigcup_k I_k)=0$. Let $T\colon X\to X$ be a map, such that $T|_{I_k}$ is $\mathcal{C}^2$ and strictly monotonic, for every interval $ I_k\in \mathcal{I} $. Assume that $T$ is conservative and ergodic {with respect to $ \lambda $}.	

\noindent The triple $(X,T,\mathcal{I})$ is a {\em basic AFN-system} if the following conditions are satisfied:
	\begin{enumerate}[label = (\roman*)]
		\item \emph{Adler's condition:} $|T''/(T')^2|$ is bounded on $\mathcal{I}$.
		\item \emph{Finite image condition:} $T\mathcal{I} = \{T(I_k):I_k\in \mathcal{I}\}$ is {a finite set}.
		\item There is a finite, nonempty subset $\mathcal{J} \subset \mathcal{I}$ of intervals $(I_{k_l})_{l=1}^j$ having one parabolic fixed point $x_l$ as an endpoint,  and {$x_l$} is a one-sided regular source, i.e. $$(x-x_l)\ T''(x)\geq 0, \quad x\in I_{k_l},\ I_{k_l}\in \mathcal{J}. $$
		\item {For $ I_{k_l}\in \mathcal{J} $, $ x\to x_l $ in $ I_{k_l} $, there is $ p_l\in\mathbb{N}^* $ and $ a_l>0 $, such that \[T(x) = x +a_l |x-z_l|^{1+p_l} + o\Big(|x-z_l|^{1+p_l}\Big).\]}
		\item The map $T$ is uniformly expanding on a set bounded away from $\{x_l : l=1,\ldots, j\}$. 
	\end{enumerate}
\end{defi}

See Figure \ref{fig-AFN} for two examples of such maps. 
 \begin{figure}[htb!]\centering
	\captionsetup[subfigure]{labelformat=empty}
	\hfill
	\begin{subfigure}[b]{0.45\textwidth}
		\includegraphics[width=\textwidth]{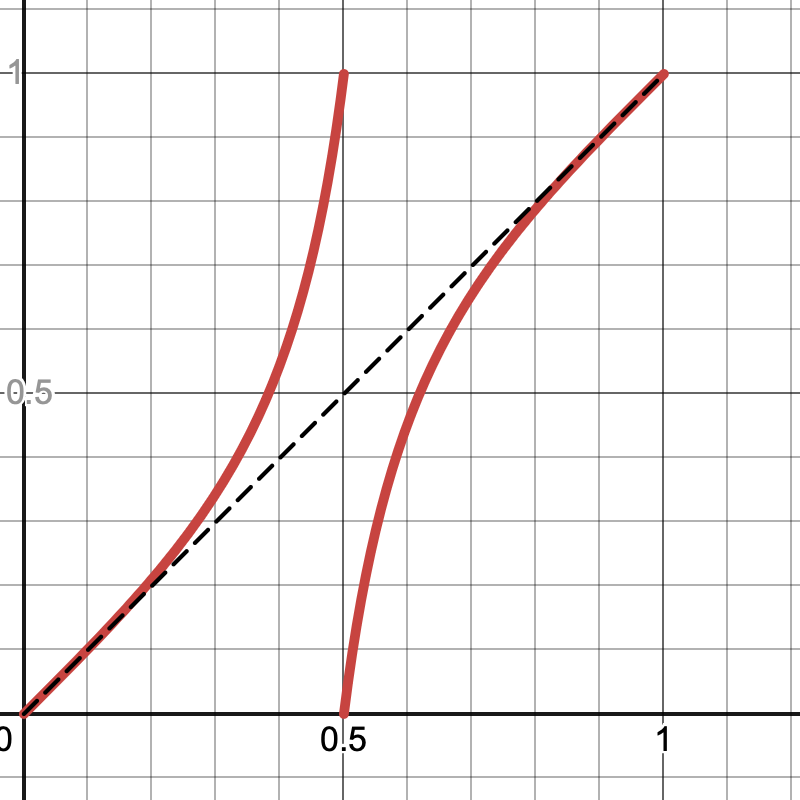}
%		\caption{\footnotesize $ x\mapsto 1/\tan(1/x) $}
		
	\end{subfigure}
	\hfill
	\hfill
	\begin{subfigure}[b]{0.45\textwidth}
		\includegraphics[width=\textwidth, ]{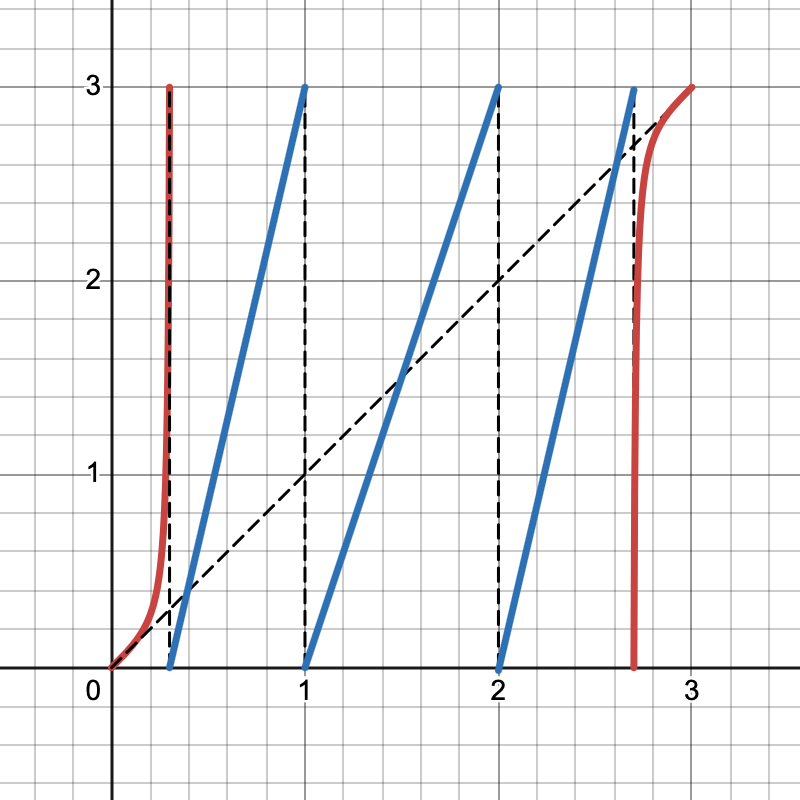}
%		\caption{\footnotesize $ x\mapsto \tan x $}
		
	\end{subfigure}
	\hfill
	\caption{{\footnotesize On the left, Boole's map (where a change of coordinates is applied to bring it to a finite interval). Such a map has two branches, both containing a parabolic point (at 0 and at 1). On the right, a AFN-system composed by five branches, two of them with a parabolic point (at 0 and at 3) --in red--, and the others hyperbolic --in blue--.}}\label{fig-AFN}
\end{figure}

{It is well-known that, if there exist $p_l \geq 1$, which is true in our situation, such maps do not preserve any finite measure absolutely continuous with respect to $ \lambda $. However, they do preserve an infinite $ \sigma $-finite measure absolutely continuous with respect to $ \lambda $. We denote such a measure by $ \mu $, which is unique (by the ergodicity of $ \lambda $)}.

The following three general results come from \cite{ThalerZweimuller}. As we will use them in a very specific situation, we do not provide all the definitions here, in particular, of the distributions $\mathcal{M}_\alpha$,  $\mathcal{L}_{\alpha,\beta}$ and $\mathcal{Z}_\alpha$ -- if necessary those can be found in the cited paper. However, under each result we will state how this translates into our setting. 

For any AFN-system, we can associate three parameters, defined below.
\begin{itemize}
	\item The parameter $p$ describes the behaviour near the parabolic point. Indeed, $p\coloneqq \max\{p_l : I_{k_l}\in \mathcal{J}\}$.
	\item The parameter $\alpha \coloneqq 1/p$. 
%	\item \anna{who is $ \beta $???- I guess $$	\beta \coloneqq \frac{\sum_{Z \in \eta,\; p_Z = p} h_0(Z)\,|a_Z|^{-\alpha}}	{\sum_{Z \in \mathcal{J},\; p_Z = p} h_0(Z)\,|a_Z|^{-\alpha}} $$ Since this depends on $ \eta $, maybe we can say that "$ \beta $ is a parameter which will be fixed later". Although then is strange to write below $ \beta=1/2 $ }
	%	\luke{ Yes, $\beta$ is defined as you say. I'm unsure how to write this best. How about this:}
	\item The parameter $\beta$ which describes the behaviour at a parabolic point with respect to all the other parabolic points. The -- rather technical -- definition is in Theorem \ref{thm:2.9}. Suffice it to say that in our situation $\beta=1/2$, basically because we have two parabolic points with the same Taylor expansion around them. 
\end{itemize}
Notably,  for our systems the three parameters used below are $p=2$, $\alpha =1/2 = \beta$ (cf. Section~\ref{sect-dpinnerfunctions} and Lemma~\ref{lem:innerAFN} -- where we prove that our system is indeed AFN).

The first result gives a weak limit version of the Ergodic Theorem, {i.e. of the asymptotic distribution of
 \[ S_nf(x)=\sum_{k=0}^{n} f\circ T^k(x).\]}
\begin{thm}{\bf (Darling--Kac theorem for AFN maps, {\normalfont\cite[Corol. 8.1]{ThalerZweimuller}})}
	Let $(X,T,\mathcal{I})$ be a basic AFN-system, and let $ E $ be a sweep-out set. Then
	\[
	\frac{1}{a_n} S_n(f) \stackrel{\mathcal{L}(\mu)}{\Longrightarrow}\int f \,d\mu\cdot \mathcal{M}_\alpha
	\quad \text{for all } f \in L^1(\mu) \text{ s.t. } \int f \,d\mu \neq 0,
	\]
	where
	\[
	a_n \sim \frac{1}{\Gamma(1+\alpha)\Gamma(2-\alpha)} \cdot \frac{n}{w_n(E)},
	\qquad n \ge 1.
	\]
\end{thm}

\begin{rema} \label{cor:11}
For $\alpha = 1/2$ the Mittag-Leffler distribution $\mathcal{M}_\alpha$ becomes $|\mathcal{N}(0,1)|$ (the absolute value of the standard normal distribution) with the density function 
\[g(x) = \frac{2}{\pi}e^{-\frac{x^2}{\pi}}\;\chi_{x>0}.\] And $a_n \sim \frac{4}{\pi}\cdot \frac{n}{ w_n(Y)}$. {And by the calculation from the proof of Theorem \ref{thm-a-in} we have that for any compact interval $E$
\[w_n(E)\sim\sum_{k=0}^{n-1} \frac{1}{\sqrt{k}}\sim \sqrt{n}. \]}
\end{rema}

The second law tells us about the number of returns into a neighbourhood of a parabolic point. {Stating it requires more notation (for formal definitions we again refer to \cite{ThalerZweimuller}). Let us abuse the notation slightly and denote the set of parabolic points as $\mathcal{J}$ (previously $\mathcal{J}$ denoted the set of parabolic intervals). Let $\mathcal{K}$ be a proper subset of $\mathcal{J}$. For every parabolic point $ x_l\in \mathcal{J} $, which is the endpoint of the interval $I_{k_l} $, we call the set $I_{k_l} \cap T^{-1}( I_{k_l} )$ the \emph{full neighbourhood} of $x_l$. Finally, denote by $A_\mathcal{K}$ the union of full neighbourhoods of points from $\mathcal{K}$.}

	Also, denote by $w_N(E,A)$ the wandering rate of $E\cap T^{-1}(A)$, and by $h_0$ the density of $\mu$ w.r.t. $\lambda$. And finally $a_{x_l}$ comes from assumption (iv) in the definition of AFN.

\begin{thm}{\bf (Arcsine law for neighbourhoods of neutral fixed points, {\normalfont\cite[Corol. 8.2]{ThalerZweimuller}})}\label{thm:2.9}
	Let $(X,T,\mathcal{I})$ be a basic AFN-system, and let $ E $ be a sweep-out set. Suppose that $\emptyset \neq \mathcal{K} \subseteq \mathcal{J}$.
	Then
	\[	\frac{w_N(E,A_\mathcal{K})}{w_N(E)}	\longrightarrow	\beta \coloneqq \frac{\sum_{x_l \in \mathcal{K},\; p_{x_l} = p} h_0(x_l)\,|a_{x_l}|^{-\alpha}}	{\sum_{x_l \in \mathcal{J},\; p_{x_l} = p} h_0(x_l)\,|a_{x_l}|^{-\alpha}}	\in [0,1]	\quad \text{as } N \to \infty,	\]
	and
	\[	\frac{1}{n} S_n(\mathbbm{1}_A)	\stackrel{\mathcal{L}(\mu)}{\Longrightarrow}	\mathcal{L}_{\alpha,\beta}	\]
	for all $A \in \mathcal{B}$ with $\mu(A \triangle A_\mathcal{K}) < \infty$.
\end{thm}

\begin{rema} \label{cor:12}
In our situation $\mathcal{J}$ contains only two elements. Thus, we will apply Theorem \ref{thm:2.9}  for $\mathcal{K}$ equal to only one of those elements. Both $h_0(x_l)$ and $a_{x_l}$ are equal at those two points (because those two parabolic points come from one \emph{split} point), so $\beta=\frac{1}{2}$. The distribution $\mathcal{L}_{\frac12,\frac12}$ is the \emph{arcsine distribution} with the cumulative distribution function given by 
\[f(t) = \frac{2}{\pi}\arcsin{\sqrt{t}}. \]
{Finally, with the notation to be introduced in Section \ref{subsection-dp-around-infty}, the set $A_\mathcal{K}$ by  definition is either equal to $(p_2^+,+\infty)$ or $(-\infty,p_2^-)$ (see Figure \ref{fig-boolep}).}
\end{rema}

We define $Z_nY(x)$ to be the time of the last visit of the orbit of $x$ to $Y$ before time $n$. The third theorem gives the behaviour of this quantity. %\anna{If $ E $ is going to be $ Y $ in the theorem, maybe we can call it $ Y $ from the begining}

%\luke{Right, let's put $E$ instead of $Y$. I think it's easier to change it here in the theorems than later in the proofs? }

\begin{thm}{\bf (Dynkin-Lamperti law for AFN reference sets, {\normalfont\cite[Corol. 8.3]{ThalerZweimuller}})}
	Let $(X,T,\mathcal{I})$ be a basic AFN-system, and let $ E $ be a sweep-out set. Then
	\[
	\frac{1}{n} Z_n(E)
	\stackrel{\mathcal{L}(\mu)}{\Longrightarrow}
	\mathcal{Z}_\alpha.
	\]
\end{thm}
\begin{rema} \label{cor:13}
For $\alpha=\frac{1}{2}$ the $Z_\alpha$ (defined as the $B(\alpha,1-\alpha)$ distribution) becomes the arcsine distribution.
\end{rema}

\begin{obs}
We would like to finish this section by observing that AFN maps are well studied and there are many more useful results concerning their statistical behaviour. For example, in \cite[Corollary 6.3.3.]{Bansard} the author proves that for $x$ not being a periodic point nor a point eventually hitting the boundary of intervals from $\mathcal{I}$ (in particular not one of the parabolic points) we have the following limit behaviour of the first entrance and first return time statistics:
\[\tau_{B(x,r)}\overset{\mathcal{L}(\mu)}{\underset{r \to 0}{\Longrightarrow}} H_{\frac{1}{2}}\Big(\Gamma\big(\frac{3}{2}\big)\Big)\mbox{\quad and \quad} \tau_{B(x,r)}\overset{\mu_{B(x,r)}}{\underset{r \to 0}{\Longrightarrow}} H_{\frac{1}{2}}\Big(\Gamma\big(\frac{3}{2}\big)\Big) \]

where $H_\alpha(\lambda)$ is a first type Mittag-Leffler law -- defined by its Laplace transform
\[\mathbb{E}\Big(e^{-sH_\alpha(\lambda)}\Big)=\frac{\lambda}{\lambda+s^\alpha}.\]
In fact the author gives the distribution for $n$-th entry/return time statistics. All of this is proved under additional assumptions of the system being \emph{mixing} and \emph{Markov}, both of which are trivially satisfied for the radial extension of one component doubly parabolic inner functions. For details, we refer to the aforementioned PhD thesis.

%\anna{question: all doubly parabolic inner functions are mixing and Markov? Or only one component? -- just out of curiosity}

%\luke{We definitely use the one-component part. This may be still true in general, but I don't think it's that trivial. If that'd be good to know, I can think on it.}
\end{obs}

\section{Doubly parabolic inner functions}\label{sect-dpinnerfunctions}

\subsection{Inner functions and one component inner functions}
\begin{defi}{\bf (Inner function)}
	A holomorphic self-map of the unit disk $ g\colon\mathbb{D}\to\mathbb{D} $ is an {\em inner function} if, for $ \lambda $-almost every point  $\xi\in \partial\mathbb{D} $, $$ g^*(\xi)\coloneqq\lim\limits_{t\to 1^-}g(t\xi)  \in\partial \mathbb{D} .$$ 
\end{defi}

 Given an inner function $ g $, we use the following notation. 
\begin{itemize}
	\item A point $ \xi\in\partial\mathbb{D} $ is called a {\em singularity} of $ g $ if $ g $ cannot be continued analytically to a neighbourhood of $ \xi $. Denote the set of all singularities of $ g $ by $ \Sigma $.
	\item We denote by $ \lambda $ the normalized Lebesgue measure on $ \partial \mathbb{D} $.
\end{itemize}

 Following Cohn \cite{Cohn} (see also \cite{ivrii2023innerfunctionscompositionoperators}), we define one component inner functions  as follows. 

\begin{defi}{\bf (One component inner function)}\label{def-ivriiurbanski-onecomp}
	An inner function $g\colon \DD\to \DD$ is a \emph{one component inner function} if for some $0<r<1$ the set $\{z\in \DD:|g(z)|<r\}$ is connected.
\end{defi}

An equivalent property, and one that is more interesting from a dynamical point of view, is that the singular values of the function 
$ g $ are compactly contained in $ \mathbb{D} $, as shown by the following result of Ivrii and Urba\'nski. 
\begin{thm}{\bf (Characterization of one component inner functions, {\normalfont \cite[Thm. 7.2]{ivrii2023innerfunctionscompositionoperators}})}\label{thm-ivriiurbanski-onecomp}
	Let $g \colon \mathbb{D} \to \mathbb{D}$ be an inner function. The following conditions
	are equivalent:
	\begin{enumerate}[label={\em (\alph*)}]
		\item The set $g^{-1}(B(0,r))$ is connected for some $0<r<1$.
		
		\item There is an annulus $A = A(0;\rho,1)$ such that
		$g \colon \mathbb{D} \to \mathbb{D}$ is a covering map over $A$.
		
		\item There is an annulus $\widetilde{A} = A(0;\rho,1/\rho)$ such that
		$g \colon \widehat{\mathbb{C}} \setminus \Sigma \to \widehat{\mathbb{C}}$
		is a covering map over $\widetilde{A}$.
		
		\item The singular set $\Sigma \subset \partial\mathbb{D}$ has Lebesgue measure $0$,
		the derivative $g'(\xi) \to \infty$ as $\xi \in \partial\mathbb{D} \setminus \Sigma$
		approaches $\Sigma$, and
		\[	\left| \frac{g''(\xi)}{g'(\xi)^2} \right| \le C,
		\qquad \xi \in \partial\mathbb{D} \smallsetminus \Sigma,\]
		for some constant $C>0$.
	\end{enumerate}
\end{thm}

\subsection{Doubly parabolic inner functions}
We now address inner functions that do not have a fixed point in $ \DD $. In particular, we focus on doubly parabolic  inner functions, defined as follows. 

\begin{defi}{\bf (Doubly parabolic inner function)}\label{def:dblpar}
	An inner function $g\colon \DD\to \DD$ is a \emph{doubly parabolic inner function} if $g$ is holomorphic in
	a neighbourhood of the Denjoy-Wolff point $p\in \partial \DD$,
	\[g(z) = p+ c(z-p)^3 + \ldots, \qquad c\neq 0,\] 
	near $p$.
\end{defi}

Note that, by definition, the Denjoy-Wolff of a doubly parabolic inner function is not a singularity for $g$ (this is not standard in the literature, but it is the approach followed in \cite{ivrii2023innerfunctionscompositionoperators}).

As above, we denote by $\Sigma$ the set of singularities of $ g $ and by $\lambda$ the Lebesgue measure on $\partial\mathbb{D}$. Moreover:\begin{itemize}
	\item Let $ \Theta $ be the set of preimages of $ p $, i.e. the set of $ \xi \in\partial\mathbb{D}$ such that $ g^*(\xi)=p $.
	\item Take a M\"obius transformation $M\colon\mathbb{D}\to\mathbb{H}$ which sends $p$ to $\infty$. Then, 
	\[
	h = M \circ g \circ M^{-1}\colon \mathbb{H}\to\mathbb{H} 
	\] 
	is a holomorphic self-map of the upper half-plane. Denote its extension to the real line by $ T\colon \mathbb{R}\to\mathbb{R} $. With a slight abuse of notation, we denote by $ \Sigma $, $ \Theta $ and $ \lambda $ the set of singularities, the set of preimages of $ \infty $, and the pushforward of $ \lambda $ to $ \mathbb{R} $, respectively.
	\item Denote by $\mu$ the Lebesgue measure on $ \mathbb{R} $. For doubly parabolic inner functions $ h\colon\mathbb{H}\to\mathbb{H} $, such a measure is $ T $-invariant (see \cite{DM91}).
\end{itemize}

Moreover, $ h\colon\mathbb{H}\to\mathbb{H} $ has the expression
\[
h(z) = z + \beta
+ \int_{\mathbb{R}} \frac{1 + z w}{w - z} \, d\rho(w),
\]
where $\beta \in \mathbb{R}$, and $\rho$ is a finite positive singular
measure on the real line \cite[Proof of Thm. 9.1]{ivrii2023innerfunctionscompositionoperators}. Differenciating, 
	\[h'(z)=1+\int_\mathbb{R}\frac{w^2+1}{(w-z)^2}\,d\rho(w).\]
	
	A {\em one component doubly parabolic} inner function is, by definition, a one component inner function (in the sense of Def. \ref{def-ivriiurbanski-onecomp}) which is also doubly parabolic. Note that such functions enjoy the properties of Theorem \ref{thm-ivriiurbanski-onecomp}.
	
	Finally, note that, when seen in the real line $ \mathbb{R} $, doubly parabolic inner functions have two parabolic branches (one at $ -\infty $ and another at $ +\infty $), and the preimages of $ \infty $ lie in a compact subset of $ \mathbb{R} $ (and there are possibly countably many of them). An example is $$ x\mapsto \frac1{\tan(1/x) } ,$$ shown in Figure \ref{fig-tangent}.

 \begin{figure}[htb!]\centering
	\captionsetup[subfigure]{labelformat=empty}
	\hfill
	\begin{subfigure}[b]{0.45\textwidth}
		\includegraphics[width=\textwidth]{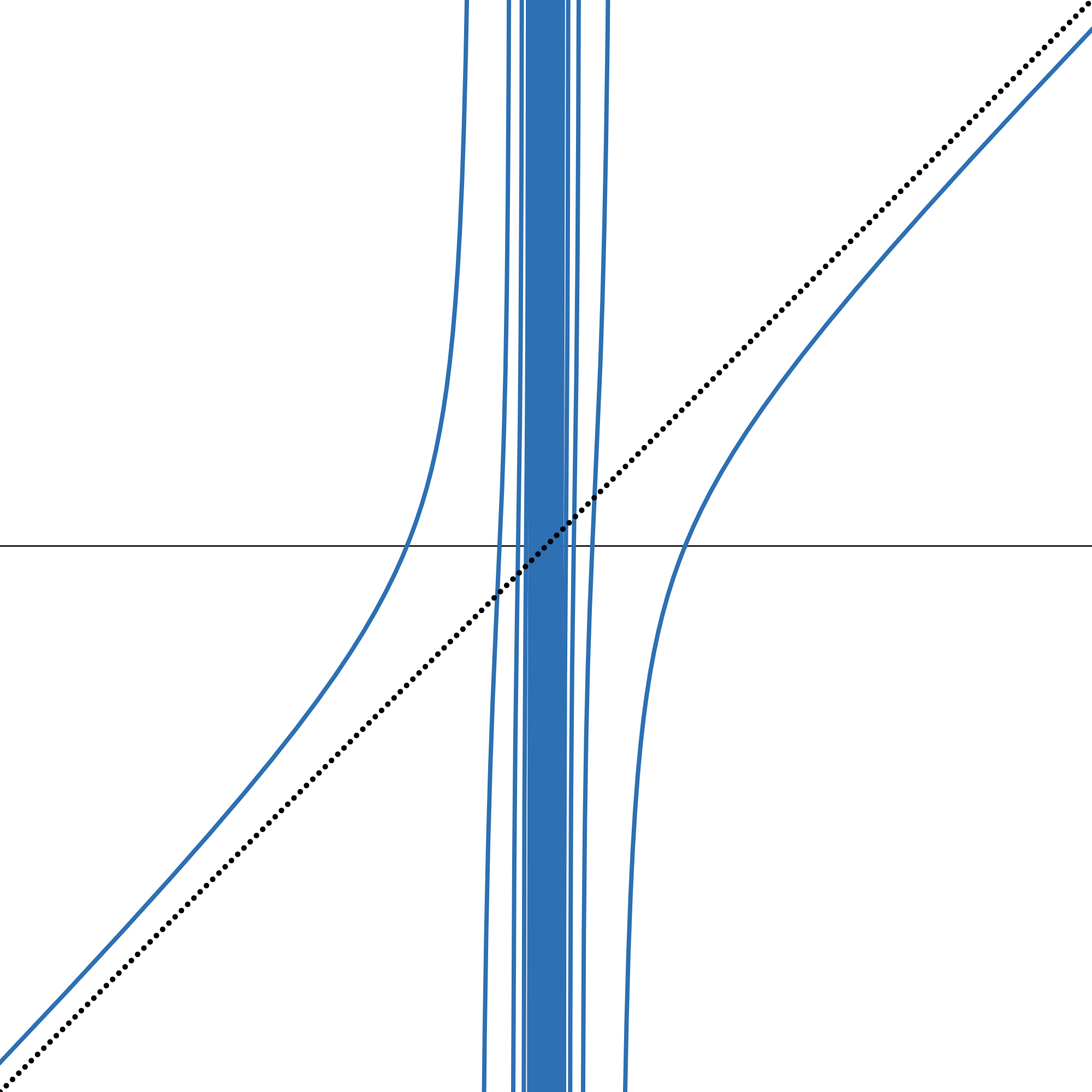}
		\caption{\footnotesize $ x\mapsto 1/\tan(1/x) $}
		
	\end{subfigure}
	\hfill
	\hfill
	\begin{subfigure}[b]{0.45\textwidth}
		\includegraphics[width=\textwidth, ]{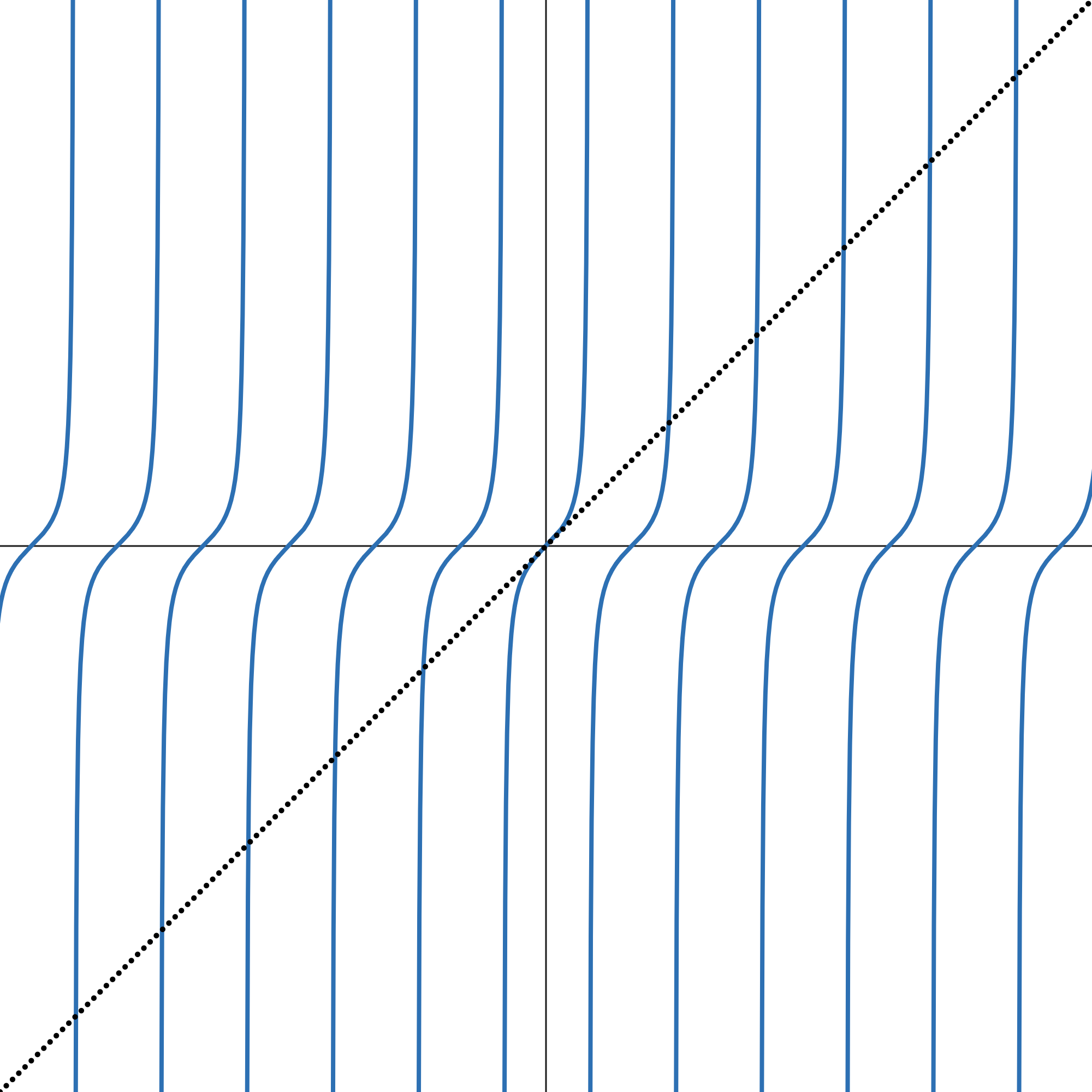}
		\caption{\footnotesize $ x\mapsto \tan x $}
		
	\end{subfigure}
	\hfill
	\caption{{\footnotesize On the left, the graph of $ x\mapsto 1/\tan( 1/x) $, which is the radial extension of a doubly parabolic inner function; 0 is a singularity of the map and $ \infty $ is the doubly parabolic fixed point. On the right, the function $ x\mapsto \tan x $ (which is conjugate to the previous function, but now the Denjoy-Wolff point is at 0 and the  singularity is at infinity).}}\label{fig-tangent}
\end{figure}

The easiest example of a doubly parabolic inner function is the degree 2 map $ h(z)=z-1/z $, whose extension $ T(x)=x-1/x $ is known as {\em Boole's map}, which has been widely studied and will play a distinguished role later on. See Figure \ref{fig-boole}.

	\begin{figure}[htb!]\centering
	\includegraphics[width=8cm]{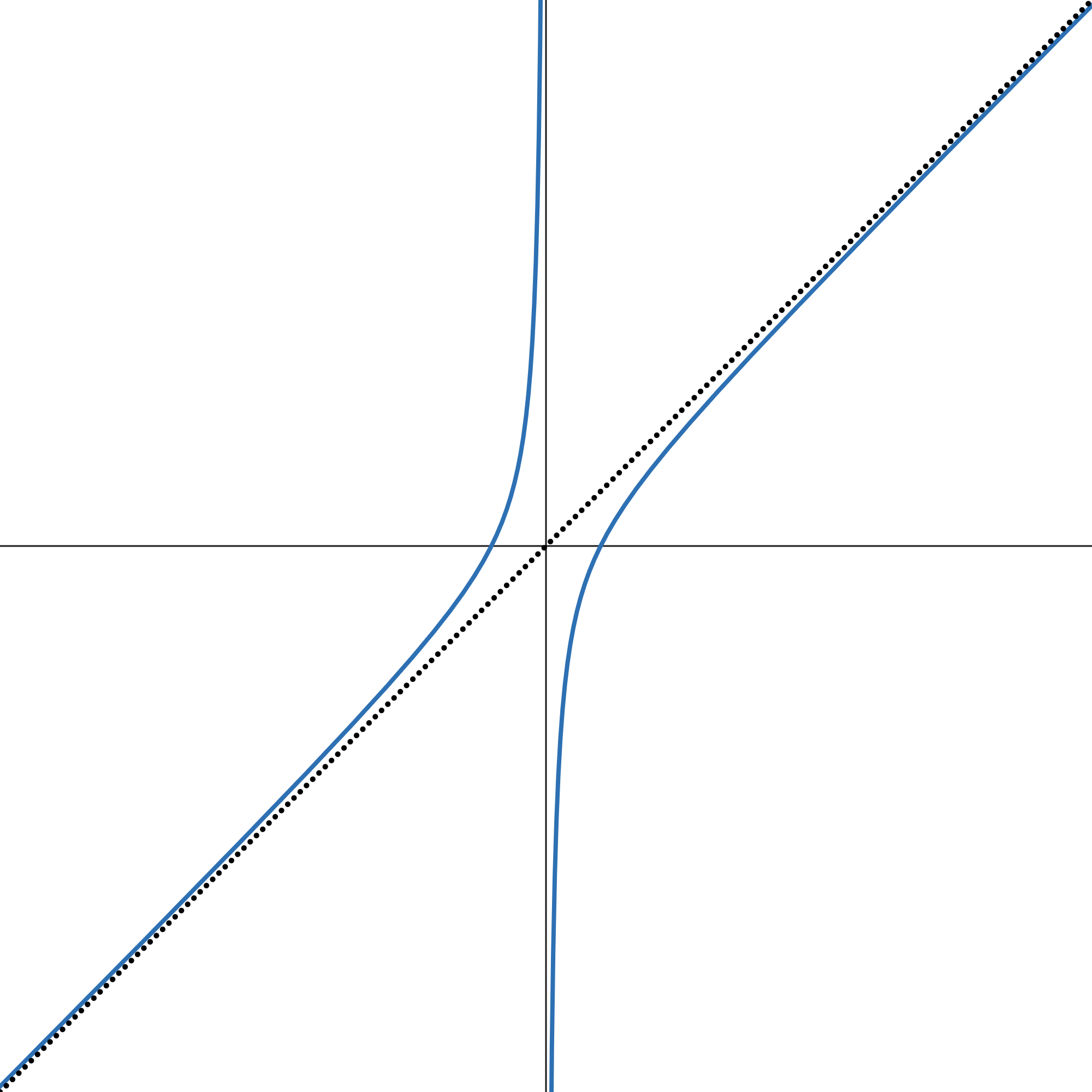}
	\setlength{\unitlength}{8cm}
	\caption{\footnotesize The graph of Boole's map $ T(x)=x-1/x $.}
	\label{fig-boole}
\end{figure}
\subsection{Mapping structure around the Denjoy-Wolff point}\label{subsection-dp-around-infty}
Most of our arguments rely on the linearization around the Denjoy-Wolff point, following the ideas of \cite{ivrii2023innerfunctionscompositionoperators}. 

More precisely, if $h$ is a doubly parabolic inner function with Denjoy-Wolff point at $ \infty $, this means that the normal form of the associated inner function around $ \infty $ point is the one corresponding to a parabolic fixed point with two invariant petals. If we place for a moment the Denjoy-Wolff point at 0, this leads to the normal form 
\[z\mapsto \phi(z) \coloneqq z+az^3+\dots\hspace{0.5cm} a>0,\] near 0 (note that, since we are dealing with inner functions, the map fixes the real line -- compare e.g. with \cite[Chap. 10]{Milnor}). By the change of coordinates $ z\mapsto 1/z $, we get
\[h(z)=\frac1{\phi(1/z)}=z-\frac{c}{z}+\dots, \hspace{0.5cm} c>0,\] near infinity.

Moreover, since $\infty$ is not a singularity of $h$, preimages of $\infty$ do not accumulate to itself. Hence, denote by $p_1^-$ and $p_1^+$ the preimages of $\infty$ which are located most to the left and most to the right, respectively. In particular, $h$ is holomorphic in $\mathbb{R}\smallsetminus\left[p_1^- ,p_1^+\right] $, and one-to-one in each of the intervals $ (-\infty, p_1^-) $ and $ (p_1^+, +\infty) $. We refer to the intervals $ (-\infty, p_1^-) $ and $ (p_1^+, +\infty) $ as the {\em parabolic branches} of $T$; while any interval delimited by two other consecutive preimages of $\infty$ is referred as a {\em hyperbolic branch} of $T$.

 Since $\infty$ is repelling, there is a  backward orbit of $ p_1^- $ in  $ (-\infty, p_1^-) $ converging to $ -\infty $ (and analogously for $ p_1^+ $), which we denote as follows
\[\dots<p_3^-<p_2^-<p_1^-\hspace{0.5cm}\textrm{and}\hspace{0.5cm} p_1^+<p_2^+<p_3^+<\dots,\] with $T(p_{n+1}^\pm)=p_n^\pm$.

We denote by $J^\pm_n$ the interval $\left[ p_n^\pm, p_{n+1}^\pm\right] $. We claim that  $p_n$ is located $ \sim \sqrt n $ away from the origin, and thus $\left| p_{n+1}^\pm - p_n^\pm \right| \sim 1/\sqrt n$.  Indeed, for the normal form with Denjoy-Wolff point at 0,  $ z\mapsto z+az^3+\dots $, $ a>0 $, points converge backwards to 0 as $ 1/ \sqrt n$ (modulo a uniform multiplicative constant which only depends on $c$) \cite[Lemma 10.1]{Milnor}. Therefore, $p_n\leq C\cdot \sqrt n$, for $n$ large enough and $ C>0 $ depending only on $c$, as desired.

Finally, note that there is a preimage of zero placed most to the left, say $z^-$, and one placed most to the right, say $z^+$.
We define $D$ as the set $(z^-, z^+)$,  which will be especially relevant in the sections that follow.

	\begin{figure}[htb!]\centering
	\includegraphics[width=12cm]{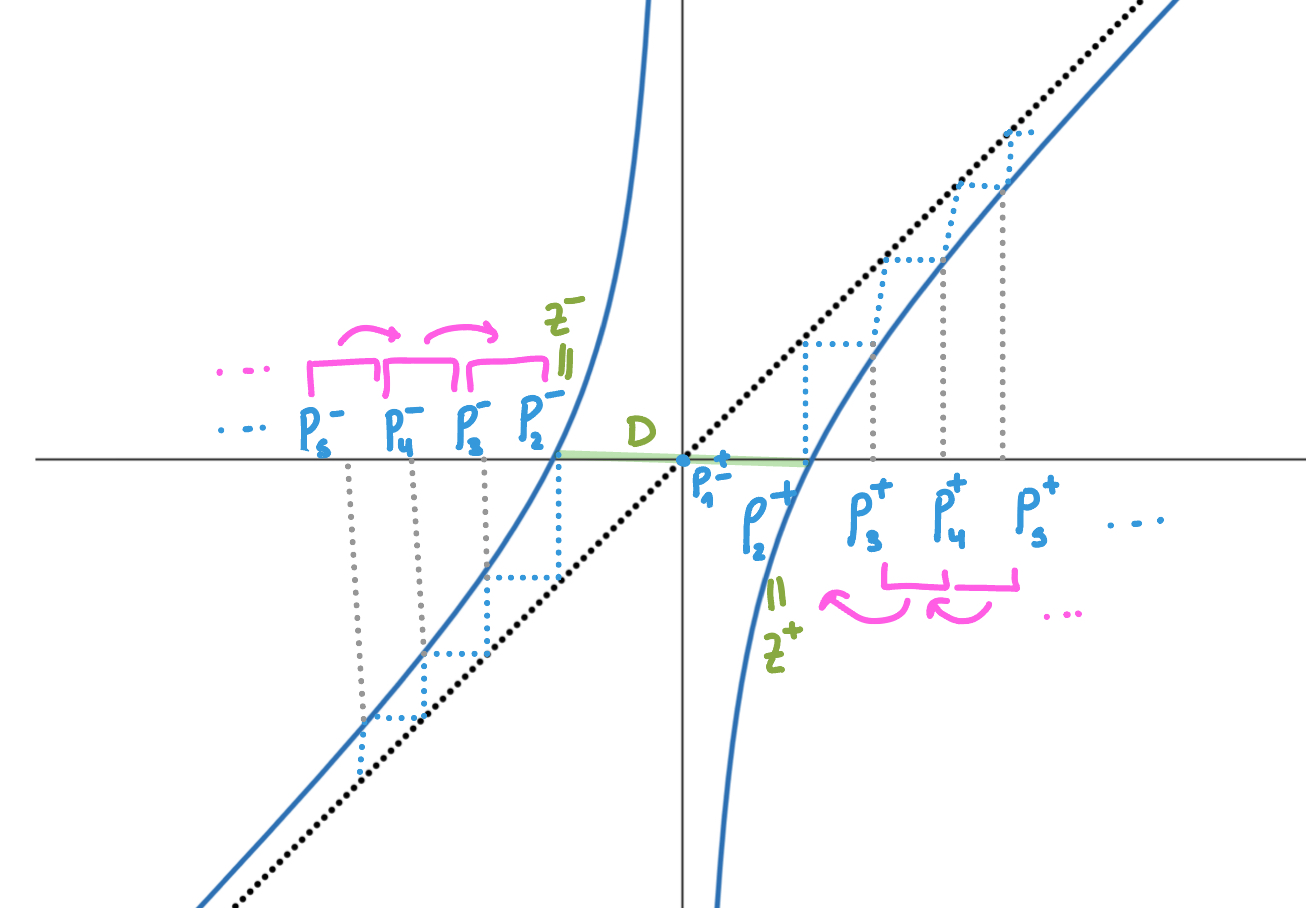}
	\caption{\footnotesize The points $ \dots<p_3^-<p_2^-<p_1^- $ and $ p_1^+<p_2^+<p_3^+<\dots $  for Boole's map $ T(x)=x-1/x $. Note that $  p_1^-$ and $  p_1^+$ coincide (and equal 0), and $ z^\pm=p_2^\pm $.}
	\label{fig-boolep}
\end{figure}

\subsection{Finite degree doubly parabolic Blaschke products}\label{subsection-finite-degree-BP}
Inner functions of finite degree are finite Blaschke products, i.e. they have the form \[g\colon\mathbb{D}\to\mathbb{D}, \ \ \ g(z)= e^{i\theta} \prod_{k=1}^{n}\frac{a_k-z}{1-\overline{a_k}z}, \] where $ \theta\in \left[ 0, 2\pi\right)  $, $ a_k\in\mathbb{D} $, and $ n\in\mathbb{N} $ is the degree of $g$. The following lemma gives a characterization of doubly parabolic finite Blaschke products when seen as functions of the upper half plane.

\begin{lemma}{\bf (Characterization of finite Blaschke products)}\label{lemma-finiteBP} The map $ h\colon\mathbb{H}\to\mathbb{H} $ is a  doubly parabolic inner function of degree $ n\in\mathbb{N} $ if and only if
	\[h(z)=z-\frac{b_1}{z-a_1}-\dots -\frac{b_n}{z-a_n},\]
for $ b_k>0 $ and pairwise different	$ a_k\in\mathbb{R} $.
\end{lemma}

Functions $ T\colon \mathbb{R}\to\mathbb{R} $ of the form \[x\mapsto x-\frac{b_1}{x-a_1}-\dots -\frac{b_n}{x-a_n},\] for pairwise different $ a_k\in\mathbb{R} $ and $ b_k>0 $, are usually called {\em generalized Boole's transformations} in the setting of infinite ergodic theory (see e.g. \cite{LiSchweiger}).
\begin{proof}
	It is clear that $ h $ is an inner function of finite degree, and that it fixes infinity. Moreover, since 
	\[h'(z)=1+\frac{b_1}{(z-a_1)^2}+\dots +\frac{b_n}{(z-a_n)^2},\]
	\[h''(z)=-\frac{2b_1}{(z-a_1)^3}-\dots -\frac{2b_n}{(z-a_n)^3},\] we have $ h'(\infty) =1$ and $ h''(\infty)=0 $, implying that $ h $ is doubly parabolic.
	
	It is left to see that any  finite degree doubly parabolic inner function can be written in the previous form. Indeed, 
	from the expression 
	\[
	h(z) = z + \beta
	+ \int_{\mathbb{R}} \frac{1 + z w}{w - z} \, d\rho(w)
	\] and the fact that $ h $ has degree $ n $, one can write \[
	h(z) = z + \beta
	+  \frac{1 + z w_1}{w_1 - z}+\dots +  \frac{1 + z w_n}{w_n - z},
	\] for $ w_1,  \dots, w_n\in \mathbb{R}$. Note that $ w_1,  \dots, w_n$ are pairwise different since $ h|_{\mathbb{R}} $ has no critical points. Then, \[ \frac{1 + z w_k}{w_k - z}=-w_k-\frac{1 + w_k^2}{z-w_k }=-w_k-\frac{b_k}{z-a_k},\] for $ a_k=w_k $ and $ b_k= 1 + w_k^2>0$. It is left to see that the additive constant $ \beta-w_1-\dots-w_n $ is 0, but this follows easily from the fact that $ h $ is doubly parabolic.
\end{proof}

Note that, if we assume $ a_1<\dots<a_n $, then $ a_1=p^-_1 $ and $ a_n=p^+_1 $ (following the notation above). Then, the parabolic branches are $ (-\infty, a_1) $ and $ (a_n, +\infty) $, while the hyperbolic ones are of the form $ (a_k, a_{k+1}) $, for $ k=1,\dots, n-1 $.

	\begin{figure}[htb!]\centering
	\includegraphics[width=8cm]{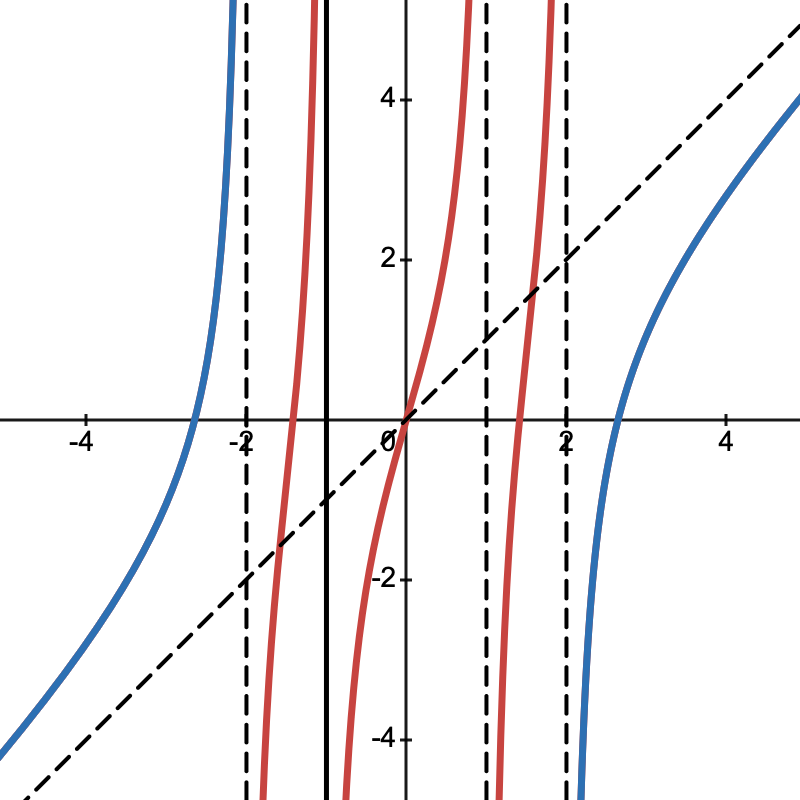}
	\caption{\footnotesize Graph of the function $ T\colon \mathbb{R}\to\mathbb{R} $, $ x\mapsto x-\frac{1}{x-1}-\frac{1}{x+1}-\frac{1}{x-2}-\frac{1}{x+2} $, which is the radial extension of a doubly parabolic Blaschke product of finite degree (seen in $ \mathbb{H} $). Its hyperbolic branches are drawn in red, while the parabolic ones in blue. Note that $ a_1=p^-_1 =-2$ and $ a_n=p^+_1 =2$.}
	\label{fig-finitebp}
\end{figure}

Recall that $D$ is the set $(z^-, z^+)$, where $z^-$ is the preimage of zero placed most to the left, and $z^+$ is the one placed most to the right. The following theorem by Li and Schweiger \cite{LiSchweiger} shows uniform bounds for the distortion of the return map $ T_D\colon D\to D $  (see \cite[Lemmas 3 and 5]{AdlerWeiss} for the same statement for Boole's map).

\begin{thm}{\bf (Distortion for the returns to $ D $, {\normalfont \cite[Lemma 3]{LiSchweiger}})}\label{lemma-LiSchweiger}
	Let $ h\colon\mathbb{H}\to\mathbb{H} $ be a doubly parabolic inner function of finite degree with Denjoy-Wolff point at $\infty$, and let $ T\colon\mathbb{R}\to \mathbb{R} $ be its extension to $ \mathbb{R} $. Let $ D=(z^-, z^+) $, and let $ T_D\colon D\to D $ be the return map to $ D $. Then, there exists a constant $ M>1 $ such that for all $ n\geq 0 $, $ D_n $   connected component of $ T_D^{-n}(D) $ and a measurable set $ E\subset D $, $ \mu(E)>0 $, we have
	\[\frac1M\leq \frac{\mu(D_n\cap T^{-n}_D(E))}{\mu(D_n)\mu(E)}\leq M.\]
\end{thm}

\begin{obs}
	The result in \cite[Lemma 3]{LiSchweiger} is stated for generalized Boole's transformations. By Lemma \ref{lemma-finiteBP}, equivalently it holds for doubly parabolic inner function of finite degree with Denjoy-Wolff point at $\infty$, as stated in \ref{lemma-LiSchweiger}.
\end{obs}

\section{Ergodic properties of the radial extension of doubly parabolic inner functions}\label{section-ergodic-DP}
We next prove some ergodic  properties of the radial extension of doubly parabolic inner functions.
\begin{thm}\label{thm-a-in}
	Let $ h\colon\mathbb{H}\to\mathbb{H} $ be a doubly parabolic inner function with Denjoy-Wolff point at $\infty$, and let $ T\colon\mathbb{R}\to \mathbb{R} $ be its extension to $ \mathbb{R} $. %Assume $ \infty $ is not a singularity of $ h $.
\begin{enumerate}[label={\em (\alph*)}]
	\item\label{thma-inside-a} Let $ E\subset\mathbb{R} $ be a compact interval, and let\[ S_nE(x)= \# \left\lbrace 0\leq k\leq n-1: h^k(x)\in E \right\rbrace =\sum_{k=0}^{n} \mathbbm{ 1 }\circ h^k(x).\]  
Then, for all $\varepsilon\in (0, \frac12)$,  \[\lim\limits_n \frac{S_nE(x)}{n^{\frac12-\varepsilon}}\to\infty, \] for $ \lambda $-almost every $ x\in \mathbb{R} $. 
\item\label{thma-inside-b} Let $ E\subset\mathbb{R} $ be a compact interval, and let \[E_n\coloneqq \left\lbrace x\in E\colon \tau_E(x)=n\right\rbrace, \] where $ \tau_E(x) \in \mathbb{N}$ denotes the first return time of $ x\in E $ to $ E $. Then, $ \lambda(E_n)\sim  \frac{1}{n\sqrt{n}} $.
\item\label{thma-inside-c} Let $ E\subset\mathbb{R} $ be a compact interval,  let $ F=\mathbb{R}\smallsetminus E $. Assume $E$ is large enough. Let \[F_n\coloneqq \left\lbrace x\in F\colon h(x), \dots, h^{n-1}(x)\in F, h^{n}(x)\notin F\right\rbrace. \] Then, $ \lambda(F_n)\sim \frac{1}{\sqrt{n}} $.
\end{enumerate}
\end{thm}
\begin{proof}
	Observe that, by Hopf Ergodic Theorem
	\ref{hopf-ergodic-thm}, in items {\em\ref{thma-inside-a}} and {\em\ref{thma-inside-b}} it is enough to prove the statement for a (well-chosen) set $E$, and then it  holds for any compact interval.

	Let us prove first {\em\ref{thma-inside-b}} and {\em\ref{thma-inside-c}}, for which we rely on the linearization around the Denjoy-Wolff point (Sect. \ref{subsection-dp-around-infty}).

	To prove {\em\ref{thma-inside-b}},  assume  $E=\left[ p^-_1, p^+_1\right] $. By Lemma \ref{lemma-zweimuller}, \[\mu (E\cap \left\lbrace \tau_E>n \right\rbrace)=\mu(E^C\cap\left\lbrace \tau_E=n \right\rbrace ),\] where $ \tau_E $ denotes the first return time to $ E $. By the definition of $p^\pm_n$, 
	\[E^C\cap\left\lbrace \tau_E=n \right\rbrace =J^-_n\cup J^+_n .\] Therefore, 
	\[\mu(E\cap \left\lbrace \tau_E>n \right\rbrace)=\mu(J^-_n\cup J^+_n )\sim \frac1{\sqrt{n}}\]
	and 
	\[\mu (E\cap \left\lbrace \tau_E=n \right\rbrace)\sim\frac1{n\sqrt{n}}.\] Since $ \lambda $ and $ \mu $ are comparable at compact intervals of the real line, \[\lambda (E\cap \left\lbrace \tau_E=n \right\rbrace)\sim\frac1{n\sqrt{n}},\]  as desired.
	Since it is enough to prove the statement for this specific set $E$, this ends the proof of {\em \ref{thmA.b}}. 
	
	To prove {\em\ref{thma-inside-c}}, assume first $E=\left[ p^-_1, p^+_1\right] $, so $ F=\mathbb{R}\smallsetminus \left[ p^-_1, p^+_1\right] $. Since $ p^\pm_2 $ are the unique preimages in $ F $ of $ p^\pm_1 $, the interval $J_1$ corresponds to the points in $ F $ which escape from $ F $ in one iteration, while points in $J_n$ are the points in $ F $ which escape from $ F $ in $n$ iterations. Thus, $$ F_n=(-\infty, p_n^-]\cup [p_n^+, +\infty) .$$
	
	Finally, we compute the $\mu$-measure of the intervals $ (-\infty, p_n^-]\cup [p_n^+, +\infty) $ (in fact, we just do it for the positive one, and we already get the desired result, which is up to multiplicative constant). Indeed,
	\[\int_{\sqrt n}^{+\infty}\frac2{x^2+1}d\mu(x)=\pi- 2\arctan\sqrt n\sim \frac 1{\sqrt n}\] showing the desired result in the case $E=\left[ p^-_1, p^+_1\right] $. Clearly, the statement holds true for any interval of the form $E=\left[ p^-_k, p^+_k\right] $, and then for any compact interval $E$ large enough.
	This ends the proof of statement {\em \ref{thmA.c}}.
	
	To prove {\em\ref{thma-inside-a}}, we work again with the set $E=\left[ p^-_1, p^+_1\right] $. As above, denote by $\tau_E$ the first return time to $E$. By Theorem \ref{thm-aaronson}, it is enough to show that, for $a(n)=n^{\frac12-\varepsilon}$, $ \varepsilon\in (0, \frac12) $, 
	\[\int_E a(\tau_E(x))d\mu(x)<\infty.\]
	Indeed, by {\em\ref{thma-inside-b}}, we have
	\[\int_E a(\tau_E(x))d\mu(x)=\sum_n a(n)\cdot \mu(E\cap \left\lbrace \tau_E=n\right\rbrace )\sim\sum_n  n^{\frac12-\varepsilon}\frac1{n\sqrt n} , \]and this later sum clearly converges.
	This ends the proof of the theorem.
\end{proof}

\subsection{AFN-systems and one component inner functions}
In this section we prove that the radial extension of a one component doubly parabolic inner function is essentially an AFN-system, which allow us to deduce finer distributional properties for its orbits.
\begin{lemma}{\bf (One component doubly parabolic inner functions are AFN)}\label{lem:innerAFN}
Let $ h\colon\mathbb{H}\to\mathbb{H} $ be a one component doubly parabolic inner function, and let $ T\colon\mathbb{R}\to \mathbb{R} $ be its radial extension to $ \mathbb{R} $. Then, there exists an analytic homeomorphism $ \varphi\colon (-\pi/2,\pi/2)\to\mathbb{R} $ such that $ ((-\pi/2,\pi/2), S, \mathcal{I}) $ is an AFN-system, where $ S \coloneqq \varphi^{-1}\circ T\circ \varphi$ and $ \mathcal{I}$ is the set of all intervals of $(-\pi/2,\pi/2)\smallsetminus \varphi^{-1}(\overline{\Theta}) $, $ \Theta=T^{-1}(\infty) $.
\end{lemma}
\begin{proof}
Take $\varphi(x)=p\tan(x)$, for $p$ large enough (the assumption on $p$ is at the end of the proof). We have the following commutative diagram:
\[
\begin{tikzcd}
	\mathbb{R} \arrow[r, "T"] 
	& \mathbb{R} \\
{	(-\frac{\pi}{2},\frac{\pi}{2})} \arrow[u, "\varphi"] \arrow[r, "S"] 
	& {(-\frac{\pi}{2},\frac{\pi}{2}) }\arrow[u, "\varphi" ']
\end{tikzcd}
\]
	
	We have to check that the system $ ((-\pi/2,\pi/2), S, \mathcal{I}) $ is indeed an AFN system, i.e. check the conditions of Definition \ref{def:AFN}. First we will show those results for the map $T$, i.e. on the real line, and then show that they are preserved by our conjugation.
	
Observe that for a parabolic inner map on the upper half-plane we have an equivalent statement of Theorem \ref{thm-ivriiurbanski-onecomp}. In particular we will use (c) and (d). 

By definition, the set $\overline{\Theta}$ is the union of $\Theta$ (i.e. the preimages of infinity) and the singular set $\Sigma$ (the accumulation points of the preimages of infinity). 
By Thm. \ref{thm-ivriiurbanski-onecomp} (d) the singular set of $T$ has Lebesgue measure $0$, and $\mathbb R \smallsetminus (\Sigma \cup \Theta)$ is a union of disjoint open intervals. On each of those intervals the map is strictly monotonic (no critical points on the boundary) and smooth. Any doubly parabolic inner map is conservative and ergodic. None of those properties change when moving from $T$ to $S$.

There are only two (thus finite and nonempty) elements in $\mathcal{J}$ with parabolic points at $+\infty$ and $-\infty$. Those are one-sided sources, giving (iii).

To get (ii), observe that on every interval (between the points of $\Sigma$) the map $T$ is onto the whole real line (thus the set of images has one element). This we know from Thm. \ref{thm-ivriiurbanski-onecomp} (c), namely that the map is covering. Moving to $S$ preserves this property (the image of every subinterval is $(-\pi/2,\pi/2)$).

Again, by (d) we have that $T''/(T')^2$ is bounded. To show the same for $S$ we need to use word-by-word the same argument as in the proof by Ivrii and Urba\'nski (part $(c) \implies (d)$) giving (i).

It remains to prove (iv). 

We start by observing that this is true for the map $T$. To see that let us recall the formulas for $T$ and its derivative from Section \ref{sect-dpinnerfunctions} (note that $\beta =0$ for a doubly parabolic map)
	\begin{align*}
	T(x) &=x + \int_{\mathbb{R}} \frac{1 + x w}{w - x} \, d\rho(w),\\
	T'(x)&=1+\int_\mathbb{R}\frac{w^2+1}{(w-x)^2}\,d\rho(w),
	\end{align*}
where $\beta \in \mathbb{R}$, and $\rho$ is some finite positive singular measure on the real line.
	
Taking a bounded set of $x$'s, gives a positive lower bound on the integral. To translate this into the map $S$ we have to use the construction of $\varphi$.
A computation of the derivative gives 
\[S'(x)=\frac{1+t^2}{1+(T(pt))^2/p^2}T'(pt),\]
where we substituted $t=\tan(x)$. Simplifying and substituting $s=pt$ yields
\[S'(x) = \frac{p^2 + s^2}{p^2+T^2(s)}T'(s).\]
Recall, the definitions of $z^+$ and $z^-$. First, assume that $s>z^+$ or $s<z^-$. Then from the behaviour of the map we know that $T$ maps $s$ towards $0$, so $s^2>T^2(s)$ and the fraction in the formula for the derivative is bigger than 1 (for any $p$).

Now take $s\in(z^-,z^+)$, let $z_{\max}= \max(|z^-|,|z^+|)$. Let $R = \rho(\mathbb{R})$. We want to prove 
\[
\frac{p^2 + s^2}{p^2+T^2(s)}T'(s) \geq K>1.
\]
So it suffices to have the following.
\begin{equation}\label{eq:pf42}\tag{$\bigtriangleup$}
p^2\frac{T'(s)}{K}-1\geq T^2(s).
\end{equation}
Using the Jensen's inequality we have (for any function $u$)
\[\Big(\int u \,d\rho\Big)^2 = \Big(\int R\cdot u \,d\frac{\rho}{R}\Big)^2 \leq \int R^2\cdot u^2 \,d\frac{\rho}{R}= R\int u^2\,d\rho.\]
Utilising it to estimate the right-hand side of \eqref{eq:pf42}  gives
\begin{align*}
T^2(s)&= \Big(s+\int_{\mathbb{R}} \frac{1 + s w}{w - s} \, d\rho(w)\Big)^2= \Big(\int_{\mathbb{R}} \frac{s
}{R}+\frac{1 + s w}{w - s} \, d\rho(w)\Big)^2\leq\\
&\leq \int_{\mathbb{R}}R\Big(\frac{1+sw(1+1/R)-s^2/R}{w-s}\Big)^2d\rho(w)\leq \int_{\mathbb{R}}\frac{(\alpha w +\beta)^2}{(w-s)^2}d\rho(w),
\end{align*}
where $\alpha$ and $\beta$ depend only on $R$ and $z_{\max}$ (as $|s|\leq z_{\max}$). On the other hand, the left-hand side of \eqref{eq:pf42} may be bounded as (assuming $p\geq \sqrt{K}$) 
\begin{align*}
\frac{p^2}{K} \Big(1+\int_\mathbb{R}\frac{w^2+1}{(w-s)^2}\,d\rho(w)\Big) -1\geq 
\frac{p^2}{K}\int_\mathbb{R}\frac{w^2+1}{(w-s)^2}\,d\rho(w).
\end{align*}
Thus it suffices to choose $p$ so big such that $\frac{p^2}{K} (w^2+1) \geq (\alpha w +\beta)^2$ for all $w$ (which is trivial to do for quadratic maps). This ends the proof of (iv) and the entire lemma.
\end{proof}

\begin{thm}\label{thm-b-in} Let $ h\colon\mathbb{H}\to\mathbb{H} $ be a one component doubly parabolic inner function, and let $ T\colon\mathbb{R}\to \mathbb{R} $ be its  extension to $ \mathbb{R} $.
	 Denote its $\sigma$-finite invariant measure by $\mu$. Let $\nu$ be any Borel probability measure $\nu \ll \mu$.
	\begin{enumerate}[label={\em (\alph*)}]
		\item\label{thmb-a-in} {\em (Darling-Kac theorem)} Let $E\subset  \mathbb{R}$ be a bounded set of finite Lebesgue measure, and let $$ S_nE(x)= \# \left\lbrace 0\leq k\leq n-1: T^k(x)\in E \right\rbrace .$$ Then, as $n\to\infty$,
		\[\nu\left(\left\lbrace \frac\pi{\sqrt {2n}} S_nE\leq \mu(E) t \right\rbrace  \right) \to \frac2\pi \int_0^t e^{-\frac{y^2}\pi }dy, \hspace{0.5cm} t\geq 0.\]
		\item\label{thmb-b-in} {\em (Arcsine law for occupation times)} 
		Let $A$ be either $(z^+,+\infty)$ or $(-\infty, z^-)$, and let $$ S_nA(x)= \# \left\lbrace 0\leq k\leq n-1: T^k(x)\in A \right\rbrace .$$	Then, as $n\to\infty$,
		\[\nu\left(\left\lbrace \frac1n S_nA\leq t \right\rbrace  \right) \to \frac2\pi \arcsin\sqrt t, \hspace{0.5cm} t\in [0,1].\]
		\item\label{thmb-c-in} {\em (Arcsine law for waiting times)} Let $E\subset  \mathbb{R}$ be a bounded set of finite Lebesgue measure, and let $Z_nE(x)$ be the time of the last visit of the orbit of $x$ to $E$ up to time $n$.
		Then, as $n\to\infty$,
		\[\nu\left(\left\lbrace \frac1n Z_nE\leq t \right\rbrace  \right) \to \frac2\pi \arcsin\sqrt t, \hspace{0.5cm} t\in [0,1].\]
	\end{enumerate}
\end{thm}

\begin{proof}
By the previous Lemma we know that $T$ is conjugated to an AFN map thus we may apply results from Subsection \ref{subsec:AFN}. By the behaviour of the map near parabolic points we know that $p=2$. It remains to apply Corollaries \ref{cor:11}, \ref{cor:12} and \ref{cor:13}.
\end{proof}

\subsection{Periodic points for doubly parabolic inner functions} The following propositions allow us to find periodic points for doubly parabolic inner functions, yielding explicit bounds for the period. We first focus on the so-called {\em Boole’s map} $ T(x)=x-\frac1x$,
which serves as a canonical model, and we then generalize the argument.
\begin{prop}{\bf (Periodic points for Boole's map)}\label{prop-periodic-Boole}
	Let $ T\colon\mathbb{R}\to \mathbb{R} $, $ T(x)= x-\frac1x $. Let $x\in \mathbb{R}$ and $ r\in(0, 1/4) $. Then, there exist a periodic point $ p\in (x-r, x+r) $ of period less than\[\left( \frac{-\ln r}{\ln2}+1\right) \cdot\frac{4}{r^2} + x^2.\]
\end{prop}

\begin{lemma}{\bf (Mapping properties for Boole's map)}\label{lemma-mapping-boole}
		Let $ T\colon\mathbb{R}\to \mathbb{R} $, $ T(x)= x-\frac1x $. Let $D=(-1,1)$.
		\begin{enumerate}[label={\em (\alph*)}]
			\item\label{lemma-mapping-boole-a} For all $ x\in \mathbb{R} $, $ T'(x)>1 $, and for all $ x\in D $, $ T'(x)>2 $.
			\item\label{lemma-mapping-boole-b} For all $ x\in \mathbb{R} $, there exists $ n=0, \dots, \lfloor x^2 \rfloor$ such that $ T^n(x)\in D $. 
			\item For all $ r\in (0,1) $, let $D_r=(-\frac{r}{2},\frac{r}{2})$. For all $ x\in D\smallsetminus D_r$, there exists $ n=1, \dots,  \lfloor 4/r^2\rfloor $ such that $ T^{n}(x)\in D $.
			\item\label{lemma-mapping-boole-d} Let $B\subset \mathbb{R}$ be a compact interval. Then, $ T^{-1}(B) $ consists of two disjoint compact intervals, not containing $ 0, 1, -1 $.
		\end{enumerate}
\end{lemma}
\begin{proof} Observe first that the image $T^n(x)$ of any given point is well-defined as long as the point is neither $0$, nor one of the countably many preimages of $0$. Note $T(D)=\mathbb R\backslash\{0\}$. Also, \[ T^2(0,1)=T(-\infty, 0)=\mathbb R\backslash\{0\},\]
	\[T^2(-1,0)=T(0,+\infty)=\mathbb R\backslash\{0\}.\]
	Since $ T $ is a proper map of degree 2, and we denote as follows the inverse branches:\[T_{1,1}\colon\mathbb{R}\longrightarrow \mathbb{R}_-\coloneqq \left( -\infty, 0\right) , \]
	\[T_{1,2}\colon\mathbb{R}\longrightarrow \mathbb{R}_+\coloneqq \left(0, +\infty\right). \] We also note that the positive (and negative analogously) can be covered by consecutive preimages (using only the branch going to the right half-line) of the set $(0,1]$. 
	
	We prove the different statements separately.
			\begin{enumerate}[label={(\alph*)}]
		\item The first item follows from the computation of the derivative: \[T'(x)=1+\frac1{x^2}.\]
		\item  Observe that from the behaviour of $T$ around infinity as a parabolic point in Section \ref{subsection-dp-around-infty} one cannot deduce the desired bound. However, since we have an explicit expression for the map, we can find the desired bound as follows. 
		
		We have that for large, positive $k$ the map needs at most $k$ steps to go from $k$ to $k-1$. More precisely, 
\[	k-1 \leq \min\{n: T^n(k)\leq k-1\}\leq k.\]
		Thus, it takes at most $$1+2+\dots+k=\frac{k^2+k}{2}\leq k^2$$ steps to get from $k$ to $(0,1)$. The proof is analogous for negative values of $ k$.
		\item  From the behaviour of the map near 0 we see that 
		$T(D_r)= (-\infty,-\frac{2}{r}+\frac{r}{2})\cup (\frac{2}{r}-\frac{r}{2},+\infty)$. Therefore, 
		\[ T(D\smallsetminus D_r)\subset \left[ -\frac{2}{r}+\frac{r}{2}, \frac{2}{r}-\frac{r}{2} \right] ,\] and applying {\em \ref{lemma-mapping-boole-b}}, we get the desired estimate.
		\item Since $T$ is a proper map of degree 2, $ T^{-1}(B) $ is compact. Moreover, $ T_{1,1} (B)\subset \mathbb{R}_-$ and  $ T_{1,2} (B)\subset \mathbb{R}_+$, so the two intervals are disjoint. It is left to show that they not contain 0, 1 nor $-1$. Indeed, if $ 0\in  T_{1,1} (B)$, then $ \infty \in B $, which contradicts the fact that $B$ is compact. One can prove the remaining part of the statement analogously, taken into account that $ T(1)=T(-1)=0 $.
	\end{enumerate}
\end{proof}

	\begin{figure}[htb!]\centering
	\includegraphics[width=10cm]{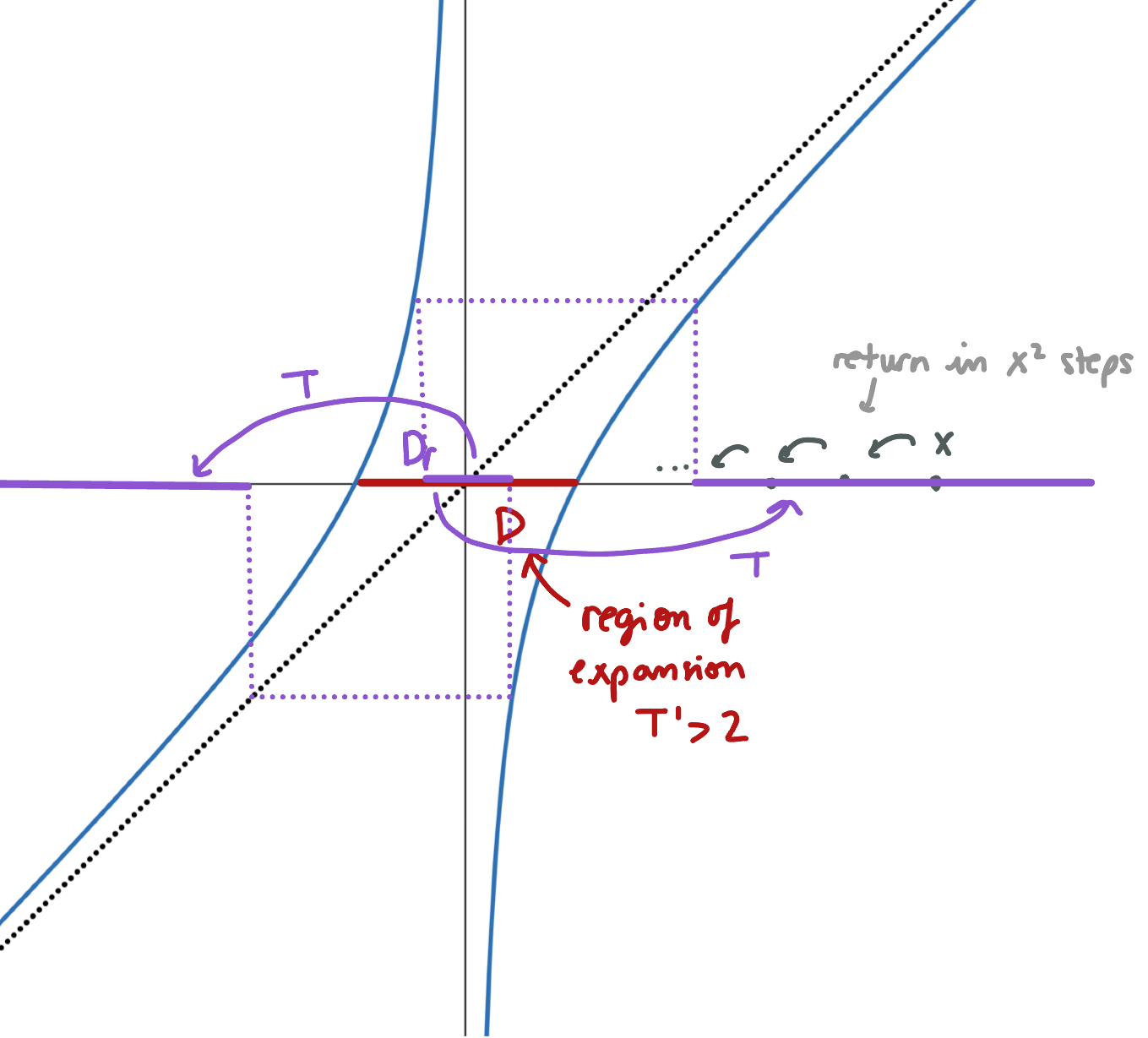}
	\caption{\footnotesize The points $ \dots<p_3^-<p_2^-<p_1^- $ and $ p_1^+<p_2^+<p_3^+<\dots $  for Boole's map $ T(x)=x-1/x $.}
	\label{fig-boole-periodic}
\end{figure}

\begin{proof}[Proof of Proposition \ref{prop-periodic-Boole}]
	The strategy of the proof is as follows. If we let $ B= (x-r, x+r)$, we want  to find the minimal $ N $ for which we can guarantee that there exists $ y\in B $, \[\#\left\lbrace n\leq N \colon T^n(y)\in D\right\rbrace>-\frac{\ln r}{\ln 2} \] and $ T^N(y)\in (x-r/2, x+r/2) $. We claim that then there is an $N$-periodic point in $B$,
	and, to complete the proof of the proposition, we will show that we can take 
	\[N=\Big\lfloor\frac{-\ln r}{\ln2}\cdot\frac{4}{r^2} + x^2\Big\rfloor.\]

First let us prove our claim on the existence of the $N$-periodic point in $B$. Consider the intervals $ B_m $, $ m=0, \dots, N $, defined as the connected component of $ T^{-m}(B) $ containg $T^n(y)$, $ n=N-m $. According to Lemma \ref{lemma-mapping-boole}{\em\ref{lemma-mapping-boole-d}}, all such intervals are well-defined, and $ 0,1,-1\notin B_m $, for all $m$. By the estimates on the derivative (Lemma \ref{lemma-mapping-boole}{\em\ref{lemma-mapping-boole-a}}), $ \left| B_{m} \right| <\left| B_{m-1} \right|$, and $2 \left| B_{m} \right| <\left| B_{m-1} \right|$ whenever $ T^n(y) \in D$ (note that if $ T^n(y) \in D$, since $ 0,1,-1\notin B_m $, then $ B_m \subset D$ and the derivative is greater than 2 for all points in $ B_m $). Since $ T^n(y) \in D $ at least $ \lfloor-\ln r/\ln 2 \rfloor$ times, and $ r\in (0,1/4) $, so $ \lfloor-\ln r/\ln 2 \rfloor>2$, we have  \[\left|  B_N\right| <\frac14\left| B\right| .\]

	Therefore, since $ T^N(y)\in (x-r/2, x+r/2) $, there exists an inverse branch of $ T^{N} $, say $ T_{N, B} $, well-defined in $B$ and $ T_{N, B} \colon B\to B$. Note that it is strictly contractive. By the Banach fixed point theorem, $ T_{N, B} $ has a fixed point in $B$, which is an $ N $-periodic point of $ T $.

Second, let us show that we can indeed take $ N $ as defined above. To start with
there are two possibilities, depending on the placement of the starting set $B$: either $x\in D$ or $x\notin D$. In the second situation we require an additional at most $x^2$ steps to enter $D$, so actually we can assume $x\in D$. We first assume that $ B\subset (0,1) $ (or analogously, $ B\subset (-1,0) $).

The goal is to show that, with $$N\leq \left(  \frac{-\ln r}{\ln 2}+1\right) \cdot \frac4{r^2},$$ then $ T^N(B)\supset (0,1)\supset B $ (and therefore there is $ y\in B $ such that $ T^N(y)\in (x-r/2, x+r/2) $). 

		As $ T'>1 $, then the diameters of the images of $B$ grow. Also those images are connected as long as they do not hit $0$, and this leads us to distinguish two cases. First, assume $ 0\notin T^n(B) $, for $ n=0, \dots, N' $, with \[N'=\Big\lfloor \frac{-\ln r}{\ln 2} \cdot \frac4{r^2} \Big\rfloor.\] Since $ 0\notin B $,  $B\cap (D\smallsetminus D_r) \neq \emptyset$, then the after at most $\lfloor 4/r^2 \rfloor$ steps we return to $D$. And because $T'|_D\geq 2$ the new image of $B$ within $D$ has radius at least $2r$, so it has to intersect $ D\smallsetminus D_r $. Proceeding in an inductive manner, by the choice of $N'$, the radius of $ T^{N'}(B) $ is at least 1, so either $ T^{N'}(B)\supset (0,1) $ or $ T^{N'}(B)\supset (-1,0) $. In any case, $ T^N(B)\subset (-1,1) $, as desired.
		
		Now assume otherwise that $ 0\in T^n(B) $, for some $ n\leq  N' $.  Then $ T^n(B) $ has to contain either $ (-r/2, 0) $ or $ (0, r/2) $  -- without loss of generality assume it is $ (-r/2, 0) $. Thus, $T^{n+1}(B)$ contains $(\frac{2}{r}+1,+\infty)$. By Lemma \ref{lemma-mapping-boole}{\em \ref{lemma-mapping-boole-b}}, the  image of this set covers $(0,1)$ after at most $4/r^2$ steps. Therefore, $ T^N(B)\supset (0,1) $ for \[N\leq N'+ \frac 4{r^2}=\left(  \frac{-\ln r}{\ln 2}+1\right) \cdot \frac4{r^2}.\]

In both cases we proved that $ T^N(B)\supset (0,1) $.

Now, assume that $ 0\in B $. Then, either $ B\cap (0,1) $ or $ B\cap (-1,0) $ have length greater that $r$, so we can apply the last argument and deduce the same bound also in this case. This ends the proof of the proposition.
\end{proof}

	Now we  turn to the general one component doubly parabolic inner functions. We want to rewrite the proof above with necessary changes.

\begin{prop}\label{prop-periodic-general}
{\bf (Periodic points for doubly parabolic inner functions)}
Let $ h\colon\mathbb{H}\to\mathbb{H} $ be a one component doubly parabolic inner function, and let $ T\colon\mathbb{R}\to \mathbb{R} $ be its radial extension to $ \mathbb{R} $. Let $x\in \mathbb{R}$ and $ r_0>0 $ small enough. Then, for all $ r\in(0, r_0) $, there exist a periodic point $ p\in (x-r, x+r) $ of period less than\[ \Big(\frac{-\ln r}{\ln K }+1\Big)\cdot \frac{W}{cr^2} +\frac{x^2}{c},\] where $ c>0$ and $ W, K>1 $ are constants that depend only on $h$.
\end{prop}

Note that, essentially, this is telling that the asymptotic rate of increase of the period as $ r\to 0 $ is \[ \frac{-\ln r}{r^2} +x^2,\] which asymptotically is the same as for Boole's map.

\begin{lemma}{\bf (Mapping properties for doubly parabolic inner functions)}\label{lemma-mapping-gen} Let $ h\colon\mathbb{H}\to\mathbb{H} $ be a  doubly parabolic inner function, and let $ T\colon\mathbb{R}\to \mathbb{R} $ be its extension to $ \mathbb{R} $. Then, there exists an  open bounded interval $D\subset \mathbb{R}$ and a set $ D_r\subset D $, consisting of a collection of open intervals of length $ r $, such that:
	\begin{enumerate}[label={\em (\alph*)}]
		\item\label{lemma-mapping-gen-a} For all $ x\in \mathbb{R} $, $ T'(x)>1 $, and for all $ x\in D $, $ T'(x)\geq K $.
		\item\label{lemma-mapping-gen-b} For all $ x\in \mathbb{R} $, there exists $ n=0, \dots, \lfloor \frac{x^2}{c} \rfloor$ such that $ T^n(x)\in D $.
		\item\label{lemma-mapping-gen-c} For all $ r\in (0,r_0) $ and $ x\in D\smallsetminus D_r$, there exists $ n=1, \dots,  \lfloor W/cr^2\rfloor $ such that $ T^{n}(x)\in D $.
		\item\label{lemma-mapping-gen-d} Let $B\subset \mathbb{R}$ be a compact interval. Then, $ T^{-1}(B) \subset \mathbb{R}\smallsetminus\Sigma$ consists of disjoint compact intervals, not containing any preimage of $ \infty $.
	\end{enumerate}
The constants $ c>0 $ and $ W,K>1 $ depend only on $h$.
\end{lemma}

\begin{proof}
	For the map $ T\colon\mathbb{R}\to\mathbb{R} $, the orbit of a point may get truncated for two reasons. First, because it hits the set $\Sigma$ of singularities of $ h $ (which is a bounded set). Second, because the point eventually hits infinity. As before, we denote the set of preimages of $ \infty $ by $\Theta$, and  let  $D$ be the set $(z^-, z^+)$, where $z^-$ is the preimage of zero placed most to the left, and $z^+$ is the one placed most to the right.
	
		The set $D_r$ is defined be covering all neighbourhoods of preimages of $\infty$, i.e.
	\[D_r = \bigcup_{s\in\Theta} B(s,r).\]
	
		\begin{enumerate}[label={(\alph*)}]
		\item Following \cite[Proof of Lemma 9.2]{ivrii2023innerfunctionscompositionoperators}, the formula for a doubly parabolic inner function of the upper half plane is 
		\[T(x)=x+\int_{\mathbb R} \frac{1+xw}{w-x}d\rho(w),\]
		for some finite positive singular measure $ \rho $ on the real line. Then, \[T'(x)=1+\int_{\mathbb R} \frac{w^2+1}{(w-x)^2}d\rho(w),\] and the statement follows.
		\item  It follows from the development of $T$ around infinity as a parabolic point, as explained  in Section \ref{subsection-dp-around-infty}. 
	%	First observation was the estimate on the number of steps required to get from $N>M$ to $D$. For this we utilise the Taylor formula at $\infty$, which we have shown in the proof of Theorem A to be 	\[T(x) = x-\frac{c}{x}+\ldots, \qquad c>0.\]Thus, the number of steps to get from $N$ to $N-1$ is approximately equal to $\frac{N}{c}$ and the number of steps required for the orbit of $N$ to enter $D$ is at most $N^2/c$. 
		\item  Use the formula for $T(x)$ above. Since $\rho$ is a finite measure, $ \rho(\mathbb{R}) <\infty$. Take $ x\in D\smallsetminus D_r $, then the distance from $x$ to the support of $ \rho $ is at least $r$ (because the support of $ \rho $ is contained in $\Sigma\cup\Theta$, which is in turn  contained in a compact interval $ \left[ -M, M\right] $). Then, 		\[\left| T(x)\right| \leq \left|  x\right| + \frac{1+\left| x\right|M }{r} \cdot \rho(\mathbb{R}).\] Since $ x\in D $, which is bounded, then  $ \left| T(x)\right| \leq  \frac Wr $, and then the estimate follows from {\em \ref{lemma-mapping-gen-b}}. 
		\item It follows from the fact that the inner function is one component (so singular values do not accumulate on the real line), and infinity is a fixed point.
	\end{enumerate}

\end{proof}
\begin{proof}[Proof of Proposition \ref{prop-periodic-general}]
Proceeding just as in the proof of Proposition \ref{prop-periodic-Boole},
  let $ B= (x-r, x+r)$, and we want  to find the minimal $ N $ for which we can guarantee that there exists $ y\in B $ such that \[\#\left\lbrace n\leq N \colon T^n(y)\in D\right\rbrace>-\frac{\ln r}{\ln K} \] and $ T^N(y)\in (x-r/2, x+r/2) $. Then there is an $N$-periodic point in $B$ (note that now the expansion in $ D $ is by $K$ and $ r $ has been taken small enough),
so we only have to show that we can take $ N $ less than  \[\Big(\frac{-\ln r}{\ln K }+1\Big)\cdot \frac{W}{cr^2} +\frac{x^2}{c}\]
	We can assume $ B\subset D $, and prove that $ N $ can be taken less than $ (\frac{-\ln r}{\ln K }+1)\cdot \frac{W}{cr^2}  $ (otherwise, iterating $ n $ times, for $ n\leq \frac{x^2}c $, $ T^n(B)\subset D $). As for Boole's map, the \emph{worst-case scenario} is when $ T^n(B) $ does not hit any singularity nor a preimage of $ \infty $ (in the first case, $ T^{n+1}(B)=\mathbb{R} $; in the second, $ T^{n+1}(B) $ contains a neighbourhood of $\infty$, and we estimate the return using Lemma \ref{lemma-mapping-gen}{\em \ref{lemma-mapping-gen-b}}). But then we know that the iterates we hit $D$ $ \lfloor \frac{-\ln r}{\ln K }\rfloor $ times,  and we can bound the return times using Lemma \ref{lemma-mapping-gen}{\em \ref{lemma-mapping-gen-c}} 
\end{proof}

\section{Boundaries of doubly parabolic Baker domains}\label{section-boundariesBD}
\subsection{Partitioning the boundary. Targets}\label{targets}
Let $U$ be a doubly parabolic Baker domain.  Let $\varphi\colon\mathbb{D}\to U$ be a Riemann map, and let $g\colon\mathbb{D}\to\mathbb{D}$, $ g\coloneqq \varphi^{-1}\circ f\circ \varphi $, be the associated inner function. It is clear that $\varphi$ is a conjugacy between $f|_U$ and $g$. 

We aim to extend this conjugacy to the boundary, i.e. between $f|_{\partial U}$ and $g^*|_{\partial\mathbb{D}}$ (the radial extension of $g$). Let us note that {\em a priori} it is unclear how to proceed, since $\varphi\colon\mathbb{D}\to U$ does not extend continuously to $\partial\mathbb{D}$ (indeed, for the considered Fatou components, its boundary is never locally connected \cite{BakerDom}), and hence it does not provide a topological conjugacy between $f|_{\partial U}$ and $g^*|_{\partial\mathbb{D}}$. The way of showing that $\partial U$ is not locally connected is by showing that points in $\partial \mathbb{D}$ whose radial limit by $\varphi$ is infinity are dense on the unit circle. Precisely these infinitely many accesses to infinity is what makes a partition of the boundary specially well-behaved with respect to the radial extension of the inner function. Next we make this argument precise.

We consider the  radial extension of the Riemann map $\varphi^*\colon\partial\mathbb{D}\to \partial U$.
The following is a classical result concerning the relation between $\varphi^*$ and the topology of the boundaries of Fatou components of entire functions, which collects the results in \cite{BakerDom, Bargmann, JF23}.

Recall that for $\xi\in\partial\mathbb{D}$, the {\em cluster set} $ Cl(\varphi, \xi) $ of $ \varphi $ at $ \xi $ is the set of values $ w\in\widehat{\mathbb{C}} $ for which there is a sequence $ \left\lbrace z_n\right\rbrace _n \subset\mathbb{D}$ such that $ z_n\to \xi $ and $\varphi (z_n)\to w $, as $ n\to\infty $.
\begin{thm}
	{\bf (Cluster sets and radial limits for Baker domains)}\label{thm-all-cluster-sets-contain-infty}
	Let $ f $ be a transcendental entire function, and let $ U $ be  a doubly parabolic Baker domain. Let $ \varphi\colon\mathbb{D}\to U $ be a Riemann map. Then,
	\begin{enumerate}[label={\em (\alph*)}]
		\item 
		$ \{\xi\in\partial\mathbb{D}\colon \varphi^*(\xi)=\infty\} $
		is dense on $ \partial \mathbb{D} $;
		\item  $ \partial U $ is the disjoint union of cluster sets $ Cl(\varphi, \cdot)$ of $ \varphi $ in $ \mathbb{C} $, i.e.
		\[\partial U= \bigsqcup\limits_{ \xi\in\partial \mathbb{D}} Cl(\varphi, \xi)\cap\mathbb{C}.\]

		\noindent Moreover,           either $ Cl(\varphi, \xi)\cap\mathbb{C}  $ is  empty, or has at most two connected components.  If $ Cl(\varphi, \xi)\cap\mathbb{C}  $ is disconnected, then $ \varphi^*(\xi)=\infty $.
	\end{enumerate}
\end{thm}

As defined in the introduction, $ E\subset \partial U $ is a {\em target} on $\partial U$ if $ E=\varphi^*(I) $, for a closed arc $I\subset \partial \mathbb{D}$, delimited by  $\xi_1\neq  \xi_2\in\partial \mathbb{D}$ such that $ \varphi^*(\xi_1)=\varphi^*(\xi_2)=\infty$.

From Theorem \ref{thm-all-cluster-sets-contain-infty}, we deduce the following basic property of targets, which determines precisely when a point is in a target,  and we can codify the orbits of points in the boundary $\partial U$ by analizing the radial extension of the associated   inner function. 

\begin{prop}{\bf (Basic property of targets)}\label{prop-targets}
	Let $E$ be a target on $\partial U$, defined as  $ E=\varphi^*(I) \subset\partial U$, for an arc $I =\left[ \xi_1, \xi_2\right] \subset \partial\mathbb{D}$ $ \varphi^*(\xi_1)=\varphi^*(\xi_2)=\infty$. Then,
	\begin{itemize}
		\item $x\in E$ if and only if $ x\in Cl(\varphi, \xi) $, $\xi \in I$;
		\item if $ x\in Cl(\varphi, \xi) $, $f^n(x)\in E $ if and only if $(g^*)^n(\xi)\in I$.
	\end{itemize} 
\end{prop}
\begin{proof}
For the first statement,	denote by $R_{\xi_i}$ the radial segment at $\xi_i$, $i=1,2$. Then, $ R_{\xi_1}\cup R_{\xi_2}\cup \left\lbrace \infty\right\rbrace  $ is a Jordan curve through infinity, which divides the Riemann sphere into two sets. One contains the cluster sets of points in $I$, and the other the cluster sets of its complement. Then, the statement follows.

For the second statement, 	note that we can think of the target $E$ as being delimited by a Jordan arc $ \gamma $ ending at $\infty$ in both directions (that is, a Jordan curve through infinity in $\widehat{\mathbb{C}}$). The preimages (by $f^n$) of such an arc are curves landing at infinity from both ends (since, for entire functions, infinity has no finite preimage). We consider only the preimages of $ \gamma $ lying in $U$. Then,  $x\in f^{-n}(E) $, and $f^{-n}(E)$ is a collection of targets delimited by the previous arcs. Applying the first statement, it is clear that $ \xi $ lies in some interval which projects by $\varphi^*$ to some target of  $f^{-n}(E)$, so $ \xi\in (g^*)^{-n}(I) $.  Then, the statement follows immediately.
\end{proof}

Finally, recall that sometimes one needs to distinguish between finite and infinite targets, defined as follows.	Let $U$ be a doubly parabolic Baker domain, and let  $\psi\colon\mathbb{H}\to U$ be a conformal map from the upper-half plane to $U$, such that points in $\mathbb{H}$ converge to $\infty$ when iterating $h\coloneqq \psi^{-1}\circ f\circ\psi$. Let $I\subset \mathbb{R}\cup \left\lbrace \infty\right\rbrace $ be a closed  interval, delimited by  $x_1\neq  x_2\in \mathbb{R}$ such that $ \psi^*(x_1)=\psi^*(x_2)=\infty$. 
\begin{itemize}
	\item If $ I$ is bounded on $\mathbb{R}$, we say that $ \psi^*(I) \subset\partial U$ is a {\em finite target} on $\partial U$. 
	\item  If the complement of  $I$ is bounded on $\mathbb{R}$, we say that $ \psi^*(I) \subset\partial U$ is a {\em infinite target} on $\partial U$.
\end{itemize}

%\section{Boundaries of doubly parabolic Baker domains}

Now we are in the position of proving the theorems concerning boundaries of doubly parabolic Baker domains stated in the introduction. Mostly, the theorems follow from the analogous ones on doubly parabolic inner functions (via Proposition \ref{prop-targets}), although for Theorem \ref{thm-c}  a finer argument is required.

\subsection{Proof of Theorems  \ref{thm-a} and \ref{thm-b}}\label{subsect-proofAB}
By Proposition \ref{prop-targets}, Theorem  \ref{thm-a} follows from Theorem \ref{thm-a-in} and Theorem  \ref{thm-b} follows from Theorem \ref{thm-b-in}. The proofs are immediate. Note that Theorem  \ref{thm-b} holds not only for Baker domains of finite degree (as stated in the introduction), but also for  Baker domains of infinite degree whose associated inner function is one component. \hfill $ \square $

\subsection{Expanding Baker domains. Theorems \ref{thm-c} and \ref{thm-dt}}
In this section, we study {\em expanding Baker domains} of entire functions, which we defined in the introduction as the ones satisfying the following conditions.
\begin{enumerate}[label={(\alph*)}]
	\item $ U $ is doubly parabolic Baker domain, and the Denjoy-Wolff point of the associated inner function is not a singularity.
	\item  For every $x, y\in\partial U $ such that $ f(y)=x$ there exists a branch $ F_y$ of $ f^{-1} $ which is well-defined in $ D(x,r) $, and
	$F_y(D(x,r)\cap U)\subset U$.
	\item $ \left| f'(z)\right| >1 $, $ z\in D(x,r) $, for all $ x\in\partial U $; and for every finite target $ E $, $ \left| f'(z)\right| \geq K_E>1$, $ z\in D(x,r) $, for all $ x\in E$.
\end{enumerate}

Note that iterated inverse branches are well-defined and conformal along all backward orbits, as in the case of expanding basins of attraction.
\begin{lemma}{\bf (Inverse branches are well-defined)}\label{lemma-baker-inverse-branch}
	Let $f$ be a transcendental entire function, and let $U$ be an {expanding basin of attraction}. Let $\left\lbrace x_n\right\rbrace _n\subset \partial U$ be a backward orbit. Then, the inverse branches $ F_n $ of $f^n$, $F_n(x_0)=x_n$ are well-defined in $D(x_0,r)$.
\end{lemma}
\begin{proof}
	The lemma follows easily from an inductive construction of the required inverse branches (compare with Lemma \ref{lemma-attr-inverse-branch}). Indeed, the existence of the first inverse branch $F_1$ is known by assumption. Note that $ F_1(D(x_0, r))\subset  D(x_1, r)$, since $ \left| f'(z)\right| >1 $, $ z\in D(x,r) $, for all $ x\in\partial U $. Since all inverse branches if $f$ are well-defined in $D(x_1, r)$, one deduces that $F_2$ is well-defined in $D(x_0, r)$. The construction of $F_n$ follows by induction.
\end{proof}

\subsection*{Proof of Theorem \ref{thm-c}}
Theorem \ref{thm-c} asserts that, given a  transcendental entire function $ f $, and $ U $  an expanding Baker domain, then  for $\widetilde{\omega_U}$-almost every backward orbit $ \left\lbrace x_n\right\rbrace _n $ and $r_0\in (0,r)$,  there exists $n_0\in\mathbb{N}$ and $ K>1 $ such that, for all $n\geq n_0$,
\[F_n(D(x_0, r_0))\subset D(x_n,  K^{-\sqrt[3]{n}} r_0),\] where $F_n$ is the branch of $f^{-n}$ sending $x_0$ to $x_n$, for some $ K>1 $. 

According to Lemma \ref{lemma-baker-inverse-branch}, we only need to prove the estimate of contraction, and it will follow from Theorem \ref{thm-a}{\em \ref{thmA.a}}. Indeed, Theorem \ref{thm-a}{\em \ref{thmA.a}} asserts, in particular, that for the finite target $E\subset\partial U$ where there is expansion by $ K_E >1$, 
 \[\lim\limits_n \frac{\# \left\lbrace 0\leq k\leq n-1: x_k\in E \right\rbrace }{{\sqrt[3]{n}}}\to\infty, \] for $ \widetilde{\omega_U} $-almost every backward orbit $ \left\lbrace x_n\right\rbrace _n\subset \partial U $. Therefore, having fixed a backward orbit $ \left\lbrace x_n\right\rbrace _n\subset \partial U $ satisfying the previous condition, we can take $ n_0\in\mathbb{N} $ such that for $n\geq n_0$,
  \[\# \left\lbrace 0\leq k\leq n-1: x_k\in E \right\rbrace > 2\sqrt[3]{n} .\]
Thus, at time $ n $, we contracted by the factor $K_E$ at least $ 2\sqrt[3]{n} $ times (and the remaining times $F_n$ do not increase distance), leading to the desired result (with $ K=K_E $).
\hfill $ \square $

\subsection*{Proof of Theorem \ref{thm-dt}}
%Theorem \ref{thm-d} asserts that, for an entire function $ f $, an expanding Baker domain $ U $ of finite degree, $x\in\partial U$ and $ r>0 $, for all $ r>0 $, there exist a periodic point $ p_r\in D(x,r) \cap \partial U$ and the period of $ p_r $ increases asymptotically as 
%\[ \frac{-\ln \omega_U(D(x,r))}{\omega_U(D(x,r))^2},\] as $ r\to 0 $.

The proof follows straight from Proposition \ref{prop-periodic-general} (and using the same notation).  

The fact that the periodic point that we found lies on $ \partial U $ follows from the assumption (b) in the definition of expanding Baker domains, and ends the proof of the statement.
\hfill $ \square $

\subsection{Examples of expanding Baker domains of any degree}\label{subsect-examples}
To end this section, let us provide examples of expanding Baker domains of any degree, inspired by the example $ f(z)=z+e^{-z} $ (studied in \cite{BakerDom, FH06, FagellaJove}, and in the next section), and the construction in \cite{FH06}.

Let \[ f\colon \mathbb{C}\to \mathbb{C}\hspace{0.5cm} f(z)=z-{P(e^{-z})},\] where $P(z)$ is a polynomial (to be determined). Note that $f(z)$ projects to a map of the punctured plane $$ F\colon \mathbb{C}^*\to \mathbb{C}^*\hspace{0.5cm}  F(w)=we^{P(w)} ,$$ via $w=e^{-z}$, i.e. the following diagram commutes.
\[ \begin{tikzcd}
	 \mathbb{C} \arrow{r}{f(z)} \arrow[swap]{d}{w=e^{-z}} &  \mathbb{C} \arrow{d}{{w=e^{-z}}} \\%
	 \mathbb{C}^* \arrow{r}{F(w)}&  \mathbb{C}^*
\end{tikzcd}
\]

Observe that, for any polynomial $P$, $F(0)=0$. Moreover, if $P$ is chosen so that $F'(0)=1$, then 0 is a parabolic fixed point, and its parabolic basin lifts to countably many doubly parabolic Baker domains, of the same degree as the parabolic basin. Since 
$ F'(0)=e^{P(0)}$, we will choose $P$ so that $ P(0)=0 $.

Note that $F$ is a map of finite type (i.e. with only finitely many singular values). It has one (finite) asymptotic value, 0, which is fixed under $F$. Critical points of $F$ satisfy the equation
\[F'(w)=e^{P(w)}(1+wP'(w))=0.\]
Then, for any $n\geq 1$, we take $P(w)$ to satisfy \[1+wP'(w)=(z-1)^n,\]i.e. so that the function $F$ has only one critical point (placed at 1), of multiplicity $n$.  We remark that the critical point is in the basin (since basins always contain a singular value).

The later equation, together with $ P(0)=0 $, determines unambiguously the polynomial $P$. 
Note that, for $n=1$ we recover the maps $f(z)=z+e^{-z}$ and $ F(w)=we^{-w} $; for $n=2$, $f(z)=z-\frac{e^{-2z}}{2}+2e^{-z}$ and $ F(w)=we^{\frac{w^2}{2}-2w} $; and so on.

Next we prove that these Baker domains are indeed expanding. For simplicity, since $f$ has always countably many Baker domains (which are the lifts of the parabolic basin), we restrict ourselves to the one containing the positive real axis (which is forward invariant -- note that its projection by $w=e^{-z}$ is the interval $ (0,1) $ which is easy to check that it is forward invariant).
\begin{prop}
The Baker domain of the function $f$ is expanding.
\end{prop}
\begin{proof}
	We shall check that the three items in the definition of expanding Baker domain hold. First, the fact that it is doubly parabolic is immediate from the argument above (being the lift of a parabolic basin), and since it is of finite degree, the Denjoy-Wolff point is not a singularity.
	
	Since the postsingular set (in this case, the orbit of the critical values) is contained in the lines $\mathbb{R}+2k\pi i$, $k\in \mathbb{Z}$ it is easy to see that there exists $r>0$ such that for every $x, y\in\partial U $ with $ f(y)=x$ there exists a branch $ F_y$ of $ f^{-1} $ which is well-defined in $ D(x,r) $. The condition 	$F_y(D(x,r)\cap U)\subset U$ follows directly from \cite[Technical Lemma]{JF23}.

It is left to prove the third item, related with $ \left| f'\right| $. Note that
\[f'(z)=1+P'(e^{-z})e^{-z}=1+ \frac{(1-e^{-z})^n-1}{e^{-z}}e^{-z}=(1-e^{-z})^n.\]
Hence, $ \left| f'\right| >1 $ if and only if $ \left| 1-e^{-z} \right| >1$. We claim that the set $  \left\lbrace z\in\mathbb{C}\colon \left| 1-e^{-z} \right| <1\right\rbrace  $ is forward invariant (and thus contained in the Fatou set). Indeed, each of the regions in the previous set contains a ray $\mathbb{R}_++2k\pi i$, $k\in \mathbb{Z}$, which is forward invariant. Moreover, since $ \left| f'(z)\right| <1 $, points in  $  \left\lbrace z\in\mathbb{C}\colon \left| 1-e^{-z} \right| <1\right\rbrace  $ are mapped closer to $\mathbb{R}_++2k\pi i$, proving the forward invariance of such set. Hence, it is clear that there exists $ r>0 $ such that 
$ \left| f'(z)\right| >1 $, $ z\in D(x,r) $, for all $ x\in\partial U $. 

We have to see that for every finite target $ E $, $ \left| f'(z)\right| \geq K_E>1$, $ z\in D(x,r) $, for all $ x\in E$. We restrict ourselves to the Baker domain containing $\mathbb{R}_+$. Note that the access to infinity which corresponds to $ \infty $ when computing the inner function in the upper-half plane (i.e. the dynamical access) is the one given by $\mathbb{R}_+$. All other accesses to infinity land to $ -\infty $ (note that, otherwise, they would correspond to other accesses to the parabolic point, and, for entire functions, points in $ \mathbb{C} $ are uniquely accessible from a Fatou component \cite[Thm. 3.14]{Bargmann}). Since $  \left| f'(z)\right| $ grows to infinity exponentially fast when $ \textrm{Re }z\to-\infty $, the existence of $ K_E $ for the finite target $ E $ follows.
\end{proof}
\subsection{
\texorpdfstring{The Baker domain of $f(z)=z+e^{-z}$ and Boole's map. Theorem \ref{thm-e}}
{The Baker domain of f(z)=z+exp(-z) and Boole's map. Theorem \ref{thm-e}}}\label{sect:exp}
Finally, we apply our results to the doubly parabolic
Baker domain of degree two of the function $ f(z)=z+e^{-z} $, studied in \cite{BakerDom, FH06,FagellaJove}.  Note that such Baker domain is expanding (in fact, it is the example in Sect. \ref{subsect-examples}, with $n=1$), so Theorems \ref{thm-c} and \ref{thm-dt} apply.

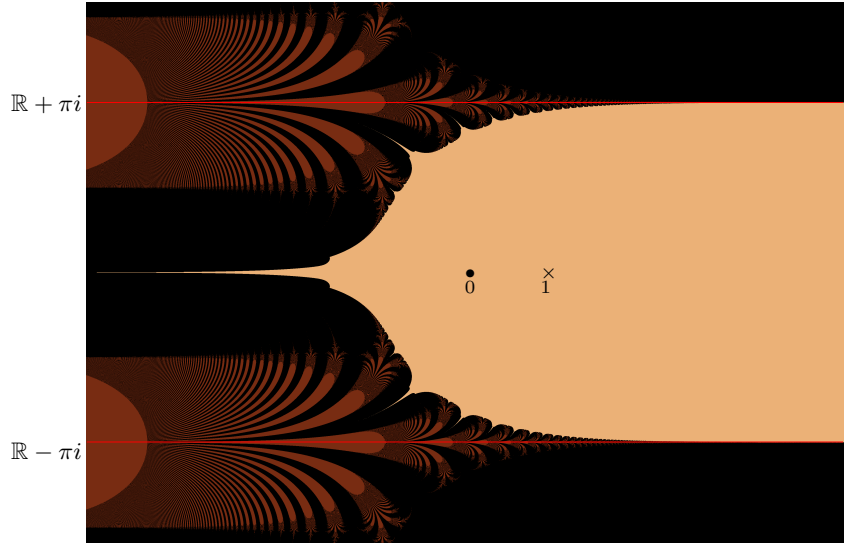
\begin{figure}[htb!]\centering
	\includegraphics[width=10cm]{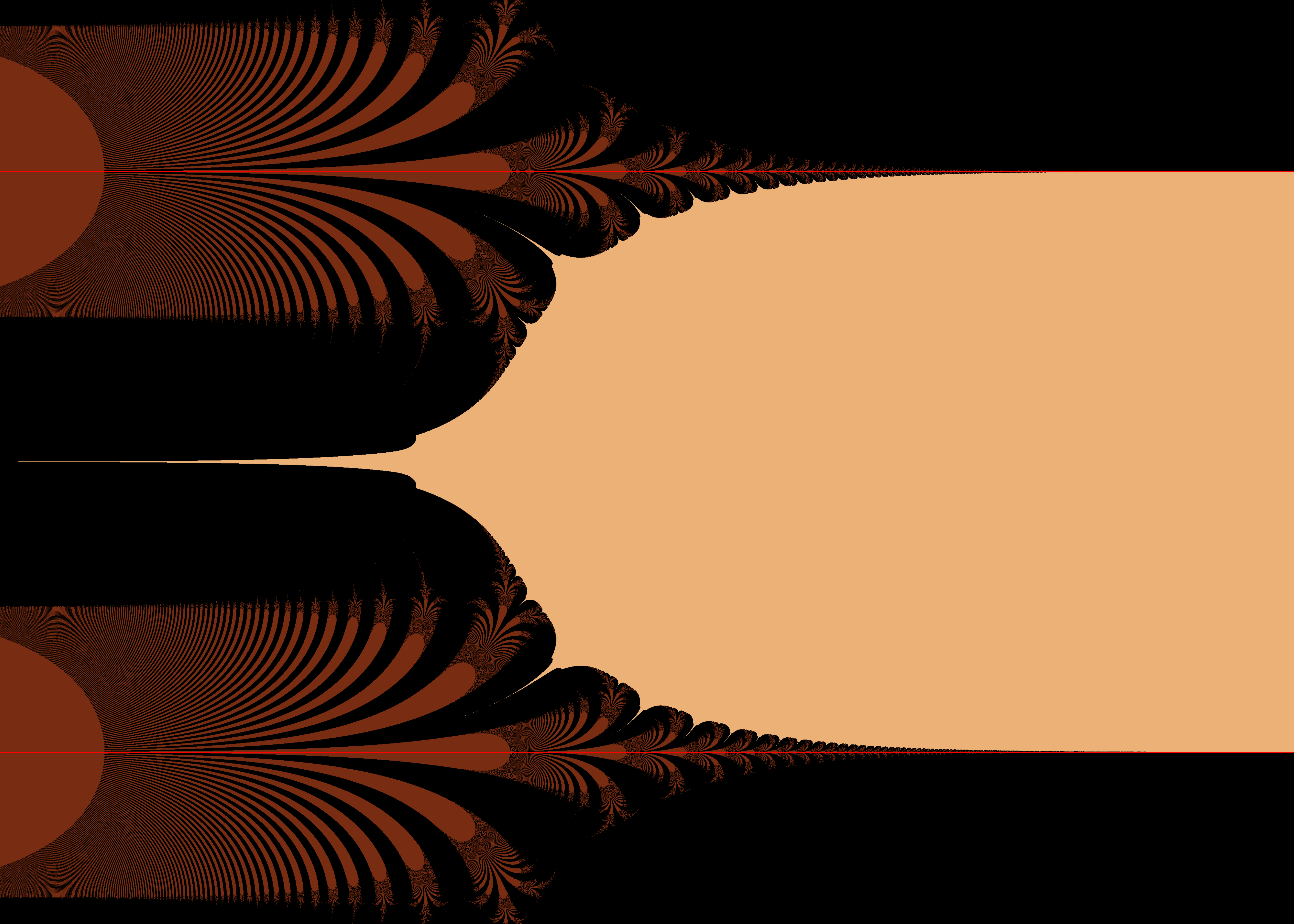}
	\setlength{\unitlength}{10cm}
	\put(-1.1, 0.57){\footnotesize $ \mathbb{R}+\pi i $}
	\put(-1.1, 0.11){\footnotesize $ \mathbb{R}-\pi i $}
	\put(-0.5, 0.35){\tiny $ \bullet$}
	\put(-0.4, 0.35){\tiny $ \times$}
	\put(-0.5, 0.33){\tiny $ 0$}
	\put(-0.4, 0.33){\tiny $ 1$}
	\caption{\footnotesize Dynamical plane for $ f(z)=z+e^{-z} $. In red, the Julia set of $ f $. In beige, the Baker domain contained in the strip $ \left\lbrace -\pi<\textrm{Im }z<\pi\right\rbrace  $. In black, the rest of the Fatou set of $ f $. The only critical point on the strip (0) is also marked, as well as the corresponding critical value (1).}
\end{figure}

 In \cite{BakerDom, FH06} general aspects of the dynamics of $f$ are discussed, including the existence of a doubly parabolic Baker domain  $U$ of degree 2, and giving a first topological description of its boundary.

In \cite{FagellaJove}, a more detailed description of the boundary is provided (both topologically and dynamically). We introduce here notation and some of the results in \cite{FagellaJove}.
\begin{itemize}
	\item Let $ \Sigma_2 $ be the space of codes of two symbols (0 and 1). We denote each sequence by $\underline{c}=\left\lbrace c_0c_1c_2...\right\rbrace $, $c_i\in \left\lbrace 0,1\right\rbrace $.
	\item Let  $ \underline{c} \in\Sigma_2$, and  let $ \underline{c}_n $ be the finite codes obtained by truncating $\underline{c}$ up to position $n$, i.e. $ \underline{c}_n =\left\lbrace c_0c_1...c_n\right\rbrace $. A {\em block} is a maximal set of consecutive 0's or 1's.
	Let us denote by $ B_n(\underline{c}) $ the number of blocks of $ \underline{c}_n $, and by $ L_n (\underline{c})$ the length of the last block of $  \underline{c}_n$.
	\item Let  $\Omega_0\coloneqq \partial U\cap \mathbb{H}_+$ and $\Omega_1\coloneqq \partial U\cap \mathbb{H}_-$, where $ \mathbb{H}_\pm $ stands for the upper and lower half-planes, respectively. Then, $ \partial U =\Omega_0\sqcup\Omega_1$. To each point $x\in\partial U$, we can give a code $ \underline{c}\in \Sigma_2 $ encoding its dynamics (i.e. conjugating the dynamics to the shift map in $ \Sigma_2 $).
	
	In \cite{FagellaJove}, it is proven that, given $ \underline{c}\in \Sigma_2 $, there exists a curve $\gamma_{\underline{c}}$ (hair) of escaping points (in fact, every escaping point with code $ \underline{c} $ is in $\gamma_{\underline{c}}$). For some $ \underline{c} $'s, it is shown that the hair lands at a finite endpoint (this set of codes have zero $\nu$-measure, to be defined below).
	\item We define a measure $\nu$ in the space $\Sigma_2$ as the push-forward of the measure $\omega_U$ on $\partial U$. Note that such a measure is finite, but it is not preserved by the map.
	On the other hand, the pushforward of the measure $\mu$ on $\partial U$ is infinite, and invariant under the shift map.
\end{itemize}  

We aim to give a more accurate description of the dynamics, using the results developed in the paper.

\begin{maintheorem}\label{thm-e}
	Let $U$ be the Baker domain of $f(z)=z+e^{-z}$. The following holds.
	\begin{enumerate}[label={\em (\alph*)}]
		\item\label{thme-a} For $\nu$-almost every code $\underline{c}\in\Sigma_2$, the hair $\gamma_{\underline{c}}$ lands at a finite endpoint.
		\item\label{thme-b}  If \[H_n\coloneqq \left\lbrace \underline{c}\in\Sigma_2\colon \textrm{ the first block of }\underline{c}\textrm{ has length }n  \right\rbrace, \] then $ \nu (H_n)\sim\frac1{\sqrt n} $, as $n\to\infty$.
		\item\label{thme-c} As $n\to\infty$,
		\[\nu\left(\left\lbrace  \underline{c}\in\Sigma_2 \colon \frac{B_n(\underline{c})}{\sqrt{n}}\leq \frac{2\sqrt2}\pi t \right\rbrace  \right) \to \frac2\pi \int_0^t e^{-\frac{y^2}\pi }dy, \hspace{0.5cm} t\geq 0,\]
		\[\nu\left(\left\lbrace \underline{c}\in\Sigma_2 \colon 1-\frac{L_n(\underline{c})}n \leq t \right\rbrace  \right) \to \frac2\pi \arcsin\sqrt t, \hspace{0.5cm} t\in [0,1].\]
	\end{enumerate}
\end{maintheorem}

Note that the first result is purely topological, but proved using techniques coming from ergodic theory. Items (b) and (c) allow us to control the lengths of the blocks (from different points of view); and hence the time spent in the lower and the upper half-planes. This kind of arguments are used recurrently in \cite{FagellaJove}, and possibly this improved results could be used to improve some of their results.
Note that item (c) is telling us that the number of different blocks in $\underline{c}$ up to position $n$ is essentially  comparable to $\frac1{\sqrt n}$;  while the length of the last block up to position $n$ is essentially  comparable to $n$.

\subsection*{The associated inner function and Boole's map} One of the aspects that makes this particular Baker domain to be especially interesting from a measure-theoretical point of view is that the boundary map $ f|_{\partial U} $ is measure-theoretically conjugate to {\em Boole's map}
\[T\colon \mathbb{R}\to \mathbb{R} \hspace{1cm}T(x)=x-\frac1x,\] which is the paradigmatic example of a transformation which preserves an infinite measure (that appeared previously in Section \ref{sect-dpinnerfunctions}). By means of the concept of target developped in the paper we can even extend this conjugacy to describe the boundary dynamics topologically.

In fact, the existence of such conjugacy holds for any doubly parabolic Baker domain of degree 2.
We prove it next.
\begin{lemma}
	Let $ U $ be a Baker domain of degree 2. Then, there exist a Riemann map $ \psi\colon\mathbb{H}\to U$ such that the associated inner function $ h=\psi^{-1}\circ f\circ \psi $ is $h(z)= z-\frac1z$.
\end{lemma}

\begin{proof}
	As in the proof of Theorem \ref{thm-a}, the map near infinity is of the form \[h(z)=z-\frac{c}{z}+\dots, \hspace{0.5cm} c>0.\] However, since the map is of degree two, we have $ h(z)=z-\frac{c}{z} $. By a conformal change of variables (i.e. by precomposing the Riemann map by an automorphism of $\mathbb{H}$), we can take $ h(z)=z-\frac{1}{z} $, as desired.
\end{proof}

If we restrict ourselves to the Baker domain of $f(z)=z+e^{-z}$, there is a priviledged Riemann map to work with, the one for which the associated inner function is $ g(z)=\frac{3z^2+1}{3+z^2} $ (compare with \cite{BakerDom, FagellaJove}). We relate this Riemann map with the conformal map $ \psi $ giving the boundary conjugacy with Boole's map.
\begin{lemma}
	Let $ U $ be the Baker domain of $ f(z)=z+e^{-z} $. $ \psi\colon\mathbb{H}\to U$ such that the associated inner function $ h=\psi^{-1}\circ f\circ \psi $ is $h(z)= z-\frac1z$. Let $ \varphi\colon\mathbb{D}\to U $ be the Riemann map so that the inner function $ g\colon\mathbb{D}\to \mathbb{D} $ associated with $ f|_U $ by $ \varphi $ is \[g(z)=\frac{3z^2+1}{3+z^2}.\] Then, $ \psi =\varphi\circ M$, where $ M(z)=i\frac{1+z}{1-z} $.
\end{lemma}
\begin{proof}
The proof is a computation.
\end{proof}

\subsection*{Proof of Theorem \ref{thm-e}}
For the proof of {\em \ref{thme-a}}, recall that the hair $ \gamma_{\underline c} $ is constructed by the concatenation of infinitely many pieces $ \gamma_n $, and \[\textrm{diam} \gamma_n\leq \textrm{diam}  F_{c_1}\circ\dots\circ F_{c_n}(D),\] where $D$ is a fixed  bounded set and $F_0$ and $F_1$ are the two branches of $f^{-1}$ fixing $U$ \cite[Thm. 4.3]{FagellaJove}. Therefore, if \[\sum_n \textrm{diam}  F_{c_1}\circ\dots\circ F_{c_n}(D)<\infty,\] the hair $\gamma_n$ lands at a finite endpoint. Note that in any finite target $f$ is strictly expanding \cite[Sect. 3]{FagellaJove}; fix such target and let $K$ be the constant of expansion. Then, invoking Theorem \ref{thm-c}, we see that for $ \nu $-almost every $ \underline{c} $, \[\textrm{diam}  F_{c_1}\circ\dots\circ F_{c_n}(D)\leq K^{-\sqrt[3]{n}} \textrm{diam}D,\] which is summable, as desired.

Items  {\em \ref{thme-b}} and  {\em \ref{thme-c}} are a reformulation of Theorems \ref{thm-a} and   \ref{thm-b} in the language of codes. Indeed, let us note that hairs are precisely cluster sets \cite[Sect. 6]{FagellaJove}, and its dynamics is governed by the inner function. Moreover, there is a distinguished target $E=\varphi^*(\left[ i, -i\right] )=\psi^*(\left[ -1, 1\right] )$, characterized by the following property. When a hair hits this target (note that all the hair hit the target at the same time, because points in the hair have all the same symbolic dynamics), it changes from being in  $\Omega_0\coloneqq \partial U\cap \mathbb{H}_+$ to  $\Omega_1\coloneqq \partial U\cap \mathbb{H}_-$ (or viceversa) in the next iteration. Hairs which are not in the target remain in $\Omega_i$. Thus, the target is the `changing zone'. In terms of codes. Thus, understanding probabilistically when hairs hit $E$ can be translated in understanding the asymptotic distribution of $H_n$, $B_n$ and $L_n$, as we show next.

By definition, $H_n$ is the set of codes whose first block is of length $n$. This is equivalent to ask for the first entrance to the set $E$ for hairs which are outside $E$ (or, equivalently, how much time remain outside $E$), and this is described by Theorem \ref{thm-a}{\em\ref{thmA.c}}.

Moreover, $B_n(\underline{c}_n)$ is the number of blocks of $\underline{c}_n$. This is precisely the number of times the hair of code $\underline{c}$ visits $E$ up to time $n$, hence its distribution is given by Theorem \ref{thm-b}{\em \ref{thmb-a}}. Finally  $L_n(\underline{c}_n)$ stands for  the length of the last  block of $\underline{c}_n$. The last block is determined by the last time we visited $E$ ($Z_nE$ in the notation of Theorem \ref{thm-b}), and 
\[L_n(\underline{c}_n)= 1-\frac{Z_nE}{n}\]and the distribution is given by Theorem \ref{thm-b}{\em \ref{thmb-c}}.

\appendix

\section{Centered Inner Functions}\label{appendix}

An inner function satisfying that $ g(0)=0 $ is called a {\em centered inner function}. For centered inner functions, the Lebesgue measure $ \lambda $ is invariant and ergodic under $ g^* $. As a consequence, the ergodic properties of $ g^* $
can be studied very easily (at least compared with doubly parabolic ones). Establishing these ergodic properties is the goal of the present section. 

Note that, by conjugating $ g $ by an appropriate automorphism of $ \DD $, we are actually considering all inner functions having a fixed point in $ \DD $.

Moreover, it is well-known that a centered inner function is uniformly
expanding on the unit circle:
\[	\inf_{\xi \in \partial\mathbb{D}} |g'(\xi)| > 1,\]
where we use the convention that $|g'(\xi)| = \infty$ if $g$ does not have
an angular derivative at $\xi$ (see \cite[Thm. 4.15]{Mashreghi}).
\subsection{Background on Ergodic Theory}\label{appendix-erg-th}
Let $ (X, \mathcal{A},\mu) $ be a measure space, and $T\colon X\to X$ be a measure-preserving transformation.  Assume $ \mu $ is a probability measure.

\begin{thm}{\bf (Birkhoff Ergodic Theorem, {\normalfont \cite[Corol. 8.2.14]{URM1}})}\label{thm-birkhoff}
	Let $T\colon X\to X$ be a measure-preserving transformation on a probability space $(X, \mathcal{A}, \mu)$. Assume $T$ is ergodic with respect to $\mu$, and $h\in L^1(\mu) $. Then, for $\mu$-almost every $x\in X$,
	\[\lim\limits_{n\to +\infty} \frac1n \sum_{k=0}^{n-1} h(T^k(x))=\int_X hd\mu.\] 
\end{thm}

\begin{thm}{\bf (Kac's Lemma, {\normalfont \cite[Lemma 10.2.6]{URM1}})}\label{theorem-kac'slemma}
	Let $T\colon X\to X$ be a measure-preserving transformation on a probability space $(X, \mathcal{A}, \mu)$. Assume $T$ is ergodic with respect to $\mu$. Let  $ E\subset X $ be a set of positive measure, and let \[E_n\coloneqq \left\lbrace x\in E\colon \tau_E(x)=n\right\rbrace, \] where $ \tau_E(x) \in \mathbb{N}$ denotes the first return time of $ x\in E $ to $ E $. Then, $$ \sum_n n\cdot \mu(E_n)= \mu(E) .$$
\end{thm}
\begin{obs}
	We notice that the statement of Theorem \ref{theorem-kac'slemma} is not exactly the one in \cite[Lemma 10.2.6]{URM1}, but it can be deduced straighforward from it. Indeed, $T$ is ergodic and invariant (and hence, conservative) with respect to $\mu$, every set of positive measure is absorbing \cite[Corol. 10.1.14]{URM1}. Since \cite[Lemma 10.2.6]{URM1} is stated in terms of absorbing sets for measure-preserving transformations (not assumed to be ergodic), we see that it holds for any set of positive measure, if we assume that the measure is finite and ergodic.
\end{obs}

\subsection{Ergodic properties of  centered inner functions}
The following theorem is analogous to Theorem \ref{thm-a}, but for centered inner functions. Notice the simplicity of its proof compared with the one of Theorem \ref{thm-a}.

\begin{thm}{\bf (Ergodic properties of centered inner functions)}\label{thm-centered-inner} Let $g\colon\mathbb{D}\to\mathbb{D}$ be a centered inner function.
	\begin{enumerate}[label={\em (\alph*)}]
		\item {\em (Occupation times)}\label{thm1.1a} For every arc $ E\subset\partial \mathbb{D} $, let \[ S_nE(x)= \# \left\lbrace 0\leq k\leq n-1: T^k(x)\in E \right\rbrace =\sum_{k=0}^{n} \mathbbm{ 1 }_E\circ T^k(x).\]  Then, \[\lim\limits_n \frac{S_nE(x)}{n}= \lambda(E) \] for $ \lambda $-almost every $ x\in \partial U $. 
		\item {\em (First return times)}\label{thm1.1b} Let $ E\subset\partial \DD $ be an arc, and let \[E_n\coloneqq \left\lbrace x\in E\colon \tau_E(x)=n\right\rbrace, \] where $ \tau_E(x) \in \mathbb{N}$ denotes the first return time of $ x\in E $ to $ E $. Then, $$ \sum_n n\cdot \lambda(E_n)= \lambda(E) .$$
		\item {\em (Waiting times)}\label{thm1.1c} Let $p\in \partial \DD$ be a regular fixed point (i.e. $ g $ is holomorphic on a neighbourhood of $ p $, and $ g(p)=p $). 	Let $ G\subset \partial\mathbb{D} $ be an open arc, $ p\in G $, and let \[G_n\coloneqq \left\lbrace \xi\in G\colon g^*(\xi), \dots, (g^*)^{n-1}(\xi)\in G\right\rbrace. \] Then, $ \lambda(G_n)\sim C^{n} $, for some constant $C$, $0<C<1$. 
	\end{enumerate}
\end{thm}

\begin{proof}
	Applying Birkhoff Ergodic Theorem \ref{thm-birkhoff} with $h=\mathbbm{ 1 }_E$, we deduce that \[ S_nE(x)= \# \left\lbrace 0\leq k\leq n-1: T^k(x)\in E \right\rbrace =\sum_{k=0}^{n} \mathbbm{ 1 }_E\circ T^k(x), \hspace{0.3cm}\lim\limits_n \frac{S_nE(x)}{n}= \lambda(E) ,\] for $ \lambda$-almost every $ x\in \partial U $, as stated in {\em \ref{thm1.1a}} . Kac's Lemma \ref{theorem-kac'slemma} gives in a direct way the estimate of the first return times given in {\em \ref{thm1.1b}} .
	
	To prove {\em \ref{thm1.1c}}, assume without loss of generality that  $ G\subset \partial\mathbb{D} $ is an open arc, $ p\in G $,  such that $ g|_G $ is holomorphic. Then, the Lebesgue measure of the arcs	\[G_n\coloneqq \left\lbrace \xi\in G\colon g(\xi), \dots, g^{n-1}(\xi)\in G\right\rbrace, \]	decays exponentially (this can be easily seen by using linearizing coordinates around the fixed point), as desired.	Note that, since $ g $ is holomorphic in $ G $, the image $ g(\xi) $ is well-defined for $ \xi\in G $, and coincides with $ g^*(\xi) $. This ends the proof of Theorem \ref{thm-centered-inner}. 
\end{proof}

\subsection{Ergodic properties of one component centered inner functions}
As usual when dealing with backward orbits, we consider {\em Rohklin's natural extension}. The set \[\widetilde{\partial \mathbb{D}}\coloneqq  \left\lbrace \left\lbrace \xi_n\right\rbrace _n\subset \partial \mathbb{D}\colon g^*(\xi_{n+1})=\xi_n, n\geq0\right\rbrace \] is the space of backward orbits, and $\widetilde{g^*}\colon\widetilde{\partial \mathbb{D}}\to \widetilde{\partial \mathbb{D}} $ is defined by $$\widetilde{g^*}(\left\lbrace \xi_n\right\rbrace _n)=g^*(\xi_0)\xi_0\xi_1...$$ Let $\pi_j(\left\lbrace \xi_n\right\rbrace _n)=\xi_j$; then, one can define a measure $\widetilde{\lambda}$ as the one satisfying $\widetilde{\lambda}(\pi_j^{-1}(A))=\lambda(A)$, for every $A\subset\partial U$ measurable. Since $\lambda$ is $g^*$-invariant and ergodic, the measure  $\widetilde{\lambda}$ is $\widetilde{g^*}$-invariant and ergodic (see e.g. \cite[Sect. 8.5]{URM1}).

Note that, for the radial extension of a one component inner function, all backward orbits are conformal, in the following sense.
\begin{lemma}{\bf (Inverse branches are well-defined)}\label{lemma-attr-inverse-branch}
	Let $g$ be a one component centered  inner function. Let $\left\lbrace \xi_n\right\rbrace _n\subset \partial \mathbb{D}$ be a bakward orbit. Then, the inverse branches $ G_n $ of $g^n$, $G_n(\xi_0)=\xi_n$ are well-defined in $D(\xi_0,r)$, for some $ r>0 $.
\end{lemma}
\begin{proof}
	The lemma follows easily from an inductive construction of the required inverse branches. Take $ r>0 $ such that there are no singular values in $ \mathbb{D}\smallsetminus D(0,r) $ (the existence of such $ r>0 $ follows from the assumption of $ g $ being one component).  Then, the existence of the first inverse branch $G_1$ in $D(\xi_0, r)$ is given. Note that $ G_1(D(\xi_0, r))\subset  D(\xi_1, r)$, since $ \left| g'(\zeta)\right| >1 $, $ \zeta\in D(\xi,r) $, for all $ \xi\in\partial \mathbb{D} $. Since all inverse branches if $g$ are well-defined in $D(\xi_1, r)$, one deduces that $G_2$ is well-defined in $D(\xi_0, r)$. The construction of $G_n$ follows by induction.
\end{proof}

It follows from \cite[Sect. 2.4]{Bargmann} that, for centered inner functions, periodic points are dense on the unit circle.
Moreover, the following proposition gives a more precise control on the dynamics, from a more quantitative point of view.
\begin{prop}{\bf (Backward orbits and periodic points)} 	Let $g$ be a one component inner centered  function.
	\begin{enumerate}[label={\em (\alph*)}]
		\item {\em (Contraction of inverse branches)}
		There exists $ K>1 $ and $ r>0 $, such that	for $\widetilde{\lambda}$-almost every backward orbit $ \left\lbrace \xi_n\right\rbrace _n $, $r_0\in (0,r)$, there exists $n_0\in\mathbb{N}$ such that, for all $n\geq n_0$,
		\[G_n(D(\xi_0, r_0))\subset D(\xi_n,  K^{-{n}} r_0),\] where $G_n$ is the branch of $g^{-n}$ sending $\xi_0$ to $\xi_n$. 
		
		\item {\em (Bound on the period of periodic boundary points)} Assume $ g $ has  degree $d\in\mathbb{N}$. Let $\xi\in\partial \mathbb{D}$ and $ r>0 $. Then, there exist a periodic point $ p\in D(\xi,r) \cap \partial \mathbb{D}$ of period at most \[\frac{-\ln r}{\ln d}  + N+1 ,\] where  $ N $ depends only on $g$.
	\end{enumerate}
\end{prop}
\begin{proof} 
	To prove (a), note that the fact that the branches $G_n$ are well-defined folows from the pevious lemma, and it is only left to prove the contraction estimates, which follows from Birkhoff Ergodic Theorem. 
	Indeed, let $E\subset\partial\mathbb{D}$  be an arc in which $ \left| g'\right| \geq K>1 $.
	Since $\widetilde{\lambda}$ is $\widetilde{g^*}$-invariant and ergodic, one can apply Birkhoff Ergodic Theorem \ref{thm-birkhoff}, as in  Theorem \ref{thm-centered-inner}{\em \ref{thm1.1a}}, to the indicator function  of  $E$, to get that for  $\widetilde{\lambda}$-almost every backward orbit $ \left\lbrace \xi_n\right\rbrace _n $
	\[\lim\limits_n \frac{\# \left\lbrace 0\leq k\leq n-1: \xi_k\in E \right\rbrace}{n}= \lambda(E) .\] Thus, $ \# \left\lbrace 0\leq k\leq n-1: x_k\in E \right\rbrace  \sim n  $. Since each time we hit the target $E$, the contraction factor is $K$, we get the desired result.
	
	To prove (b), note that $g|_{\partial\mathbb{D}}$ is conjugate to $ z\mapsto z^d $ on $\partial \mathbb{D}$, and $ \left| g'\right| \geq K>1 $ on $ \partial\mathbb{D} $. Let $N$ be such that $K^N>4$.
	For $$n=\frac{-\ln r}{\ln d} ,$$ we have \[\lambda(g^n(D(\xi,r)))=1.\] In particular, there exists $y\in \partial \mathbb{D}\cap D(\xi,r/2)$ such that   $ g^{n+N+1}(y)\in D(y, r/4)$ (this follows from the conjugacy on $\partial \mathbb{D}$ to $ z\mapsto z^d $ and the fact that $ \lambda(f^n(D(\xi,r)))=1 $, so all its images have also full measure, meaning that there are preimages of all points in $ D(\xi,r) $). 
	
	Then, let $G_{n+N+1}$ be the inverse branch of $g^{n+N+1}$ sending $g^{n+N+1}(y)$ to $y$, which is well-defined in $ D(\xi,r) $. Moreover,  since $g^{n+N+1}(y)\in D(y, r/4)$ and $K^N>4$, it follows that $$G_{n+N+1}(D(\xi, r))\subset (D(G_{n+N+1}(\xi), r/4))\subset D(G_{n+N+1}(\xi), r/2)\subset D(\xi, r).$$ This already implies that there is a periodic point of period in $ D(\xi, r)$, as desired.
\end{proof}

\bibliographystyle{amsalpha}
\bibliography{ref}

\end{document}